\documentclass[11pt, reqno]{amsart}
\usepackage{amsmath,amsopn,amssymb,amsthm,multicol}
\usepackage[cp1251]{inputenc}
\usepackage[english]{babel}
\usepackage[breaklinks=true,colorlinks=true,linkcolor=blue,citecolor=blue,urlcolor=blue]{hyperref}
\usepackage[dvipdf]{graphicx}

\renewenvironment{proof}[1][Proof]{\textbf{#1.} }
{\ \rule{0.5em}{0.5em}}

\DeclareMathOperator{\Ad}{Ad}

\DeclareMathOperator{\Ric}{Ric}

\DeclareMathOperator{\Id}{Id}

\DeclareMathOperator{\Vol}{Vol}

\renewcommand{\arraystretch}{1.5}

\newtheorem{theorem}{Theorem}
\newtheorem{proposition}{Proposition}
\newtheorem{lemma}{Lemma}
\newtheorem{remark}{Remark}

\begin{document}

\title[The evolution of positively curved invariant Riemannian metrics on \dots]
{The evolution of positively curved invariant Riemannian metrics on the Wallach spaces under the Ricci flow}

\author{N.\,A.~Abiev}
\address{N.\,A.~Abiev \newline
M.\,Kh.~Dulaty Taraz State University, Taraz, Tole bi st., 60, 080000, Kazakhstan}

\email{abievn@mail.ru}

\author[Yu.\,G.~Nikonorov]{Yu.G.~Nikonorov}

\address{Yu.\,G.~Nikonorov \newline
Southern Mathematical Institute of the Vladikavkaz Scientific Center of the Russian Academy of Sciences,
Vladikavkaz, Markus st., 22, 362027, Russia}
\email{nikonorov2006@mail.ru}

\begin{abstract} This paper is devoted to the study of the evolution of positively curved metrics on the Wallach spaces
$SU(3)/T_{\max}$, $Sp(3)/Sp(1)\times Sp(1)\times Sp(1)$, and $F_4/Spin(8)$.
We prove that for all Wallach spaces, the normalized Ricci flow evolves all generic invariant Riemannian metrics with positive sectional curvature
into metrics with mixed sectional curvature.
Moreover, we prove that for the spaces $Sp(3)/Sp(1)\times Sp(1)\times Sp(1)$ and $F_4/Spin(8)$, the normalized Ricci flow evolves
all generic invariant Riemannian  metrics with positive Ricci curvature
into metrics with mixed Ricci curvature. We also get similar results for some more general homogeneous spaces.

\vspace{2mm} \noindent Key words and phrases: Wallach space, generalized Wallach space,
Riemannian metric, Ricci curvature, Ricci flow, scalar curvature, sectional curvature, planar dynamical system, singular point.

\vspace{2mm}

\noindent {\it 2010 Mathematics Subject Classification:} 53C30 (primary), 53C44, 37C10, 343C05 (secondary).
\end{abstract}

\maketitle

\section*{Introduction and the main results}\label{vvedenie}

The study of Riemannian manifolds with positive sectional curvature has a long history.
There are very few known examples, many of them are homogeneous. Homogeneous Riemannian manifolds
consist, apart from the rank one symmetric spaces, of certain homogeneous spaces in dimensions
6, 7, 12, 13 and 24 due to Berger \cite{Be}, Wallach \cite{Wal}, and Aloff--Wallach \cite{AW}.
The homogeneous spaces which admit homogeneous metrics with positive sectional
curvature have been classified in \cite{BB,Be,Wal}.
As was recently observed by J.~A.~Wolf and M.~Xu \cite{XuWolf}, there is a gap in
B\'erard Bergery's classification of odd dimensional positively curved  homogeneous spaces in
the case of the Stiefel manifold $Sp(2)/U(1)=SO(5)/SO(2)$. A refined proof of the suitable result was obtained by B.~Wilking, see Theorem 5.1 in \cite{XuWolf}.
The recent paper \cite{WiZi} by B.~Wilking and W.~Ziller gives a new and short proof of the classification of homogeneous manifolds of positive curvature.
A detailed exposition of various results on the set of invariant metrics with positive sectional curvature,
the best pinching constant, and full connected isometry groups
could be found in the papers \cite{Puttmann,Sh2,Valiev,VerZi,Volp1,Volp2,Volp3}.
\smallskip

It is a natural type of problems to investigate whether or not the positiveness of the sectional curvature or positiveness of the Ricci curvature
is preserved under the Ricci flow \cite{Bes,Ham}.
A~recent survey on the evolution of positively curved Riemannian metrics under the Ricci flow could be found in
\cite{Ni}. Interesting results on the evolution of invariant Riemannian metrics could also be found in
the papers \cite{Boe,Bo,Bu,Jab,Laf,Lauret,Payne,Wal} and the references therein.
Sometimes it is helpful to use the (volume) normalized Ricci flow, see details e.~g. on pp.~259--260 of~\cite{Ham}.
The main object of our study in this paper are {\it the Wallach spaces}
{\renewcommand{\arraystretch}{1.5}
\begin{equation} \label{SWS}
\begin{array}{l}
W_6:=SU(3)/T_{\max}, \\
W_{12}:=Sp(3)/Sp(1)\times Sp(1)\times Sp(1), \\
W_{24}:=F_4/Spin(8)
\end{array}
\end{equation}
}
\!that admit invariant Riemannian metrics of positive sectional curvature \cite{Wal}. Note that the Wallach spaces are the total spaces of the following submersions:
$S^2\rightarrow W_6 \rightarrow \mathbb{CP}^2$,
$S^4\rightarrow W_{12} \rightarrow \mathbb{HP}^2$,
$S^8\rightarrow W_{24} \rightarrow \rm{Ca}\mathbb{P}^2$.
The present paper is devoted to a detailed analysis of evolutions of positively curved metrics under the normalized Ricci flow
on all Wallach spaces.
\smallskip

On the given Wallach space $G/H$, the space of invariant metric  depends on
three positive parameters $x_1$, $x_2$, $x_3$ (see \eqref{metric} below).
The subspace of invariant metrics
satisfying $x_i=x_j$ for some $i\neq j$,
is invariant under the normalized Ricci flow, because these special metrics have a larger connected isometry group.
Indeed, such a metric $(x_1,x_2,x_3)$ admits additional isometries generated by
the right action of the group $K\subset G$ with the Lie algebra $\mathfrak{k}:=\mathfrak{h}\oplus \mathfrak{p}_k$, $\{i,j,k\}=\{1,2,3\}$,
see details in \cite{Nikonorov4}.
All such metrics are related to the above mentioned submersions of the form $K/H \rightarrow G/H \rightarrow G/K$,
coming from inclusions $H \subset K \subset G$, see e.~g. \cite[Chapter 9]{Bes}.
In what follows we call these metrics {\it exceptional} or {\it submersion metrics}.
These metrics constitute three one-parameter families up to a homothety.
All other metrics  we call {\it generic} or {\it non-exceptional}.
Our first main result is the following

\begin{theorem}\label{the2} On  the   Wallach spaces $W_6$, $W_{12}$, and $W_{24}$,  the normalized Ricci flow
evolves all generic metrics with positive sectional curvature into metrics with mixed sectional curvature.
\end{theorem}

Moreover, we will show that the normalized Ricci flow removes every generic metric from the set of metrics with positive sectional
curvature in a finite time and does not return it back to this set. This finite time depends of the initial points and could be as long as we want,
see details in Section \ref{sectionalm}.
\smallskip

Theorem \ref{the2} easily implies the following result obtained in \cite{ChWal}:
on the   Wallach spaces
$W_6$, $W_{12}$, and  $W_{24}$,  the normalized Ricci flow
evolves some metrics with positive sectional curvature into metrics with mixed sectional curvature.
Our second main result is related to the evolution of metrics with positive Ricci curvature.

\begin{theorem}\label{Sect_Ricci}
On  the   Wallach spaces $W_{12}$ and $W_{24}$,  the normalized Ricci flow
evolves all generic metrics  with positive Ricci curvature into metrics with mixed Ricci curvature.
\end{theorem}

Moreover, the normalized Ricci flow removes every generic metric from the set of metrics with positive Ricci
curvature in a finite time and does not return it back to this set. This finite time depends of the initial points and could be as long as we want.
Note also that the normalized Ricci flow can evolve some metrics with mixed Ricci curvature
to metrics with positive Ricci curvature. Moreover, there is a non-extendable integral curve of the normalized Ricci flow with exactly
one metric of non-negative Ricci curvature,
see details in Section~\ref{on_R_set}.

In the paper \cite{Bo}, C.~B\"ohm and B.~Wilking
studied (in particular) some properties of the (normalized) Ricci flow on the Wallach space $W_{12}$.
They proved that the (normalized) Ricci flow
on  $W_{12}$ evolves certain positively curved metrics into metrics with mixed Ricci curvature (see Theorem 3.1 in \cite{Bo}).
The same assertion for the space $W_{24}$ obtained by Man-Wai~Cheung and N.\,R.~Wallach in \cite{ChWal} (see Theorem 3 in \cite{ChWal}).
On the other hand, it was proved in Theorem 8 of~\cite{ChWal} that every invariant metric with  positive sectional curvature
on the space $W_6$ retains  positive Ricci curvature under the Ricci flow.
Hence, Theorem \ref{Sect_Ricci} fails for $W_6$.
Note also that for some  invariant metrics with positive Ricci curvature on $W_6$,
the Ricci flow can evolve  them to metrics with mixed Ricci curvature,
see Theorem~3 in~\cite{ChWal} or Remark~\ref{obrdornet2} below.
The principal distinction between $W_6$ and two other Wallach spaces is explained in Lemma \ref{attains_r_i} and Remark \ref{exccase}.
We emphasize that the special status of $W_6$ follows from Proposition~\ref{asymp3} and the description
of the boundary of $R$, the set of metrics with positive Ricci curvature
\eqref{r_ir}.

\begin{figure*}[t]
\centering
\includegraphics[angle=0, width=0.45\textwidth]{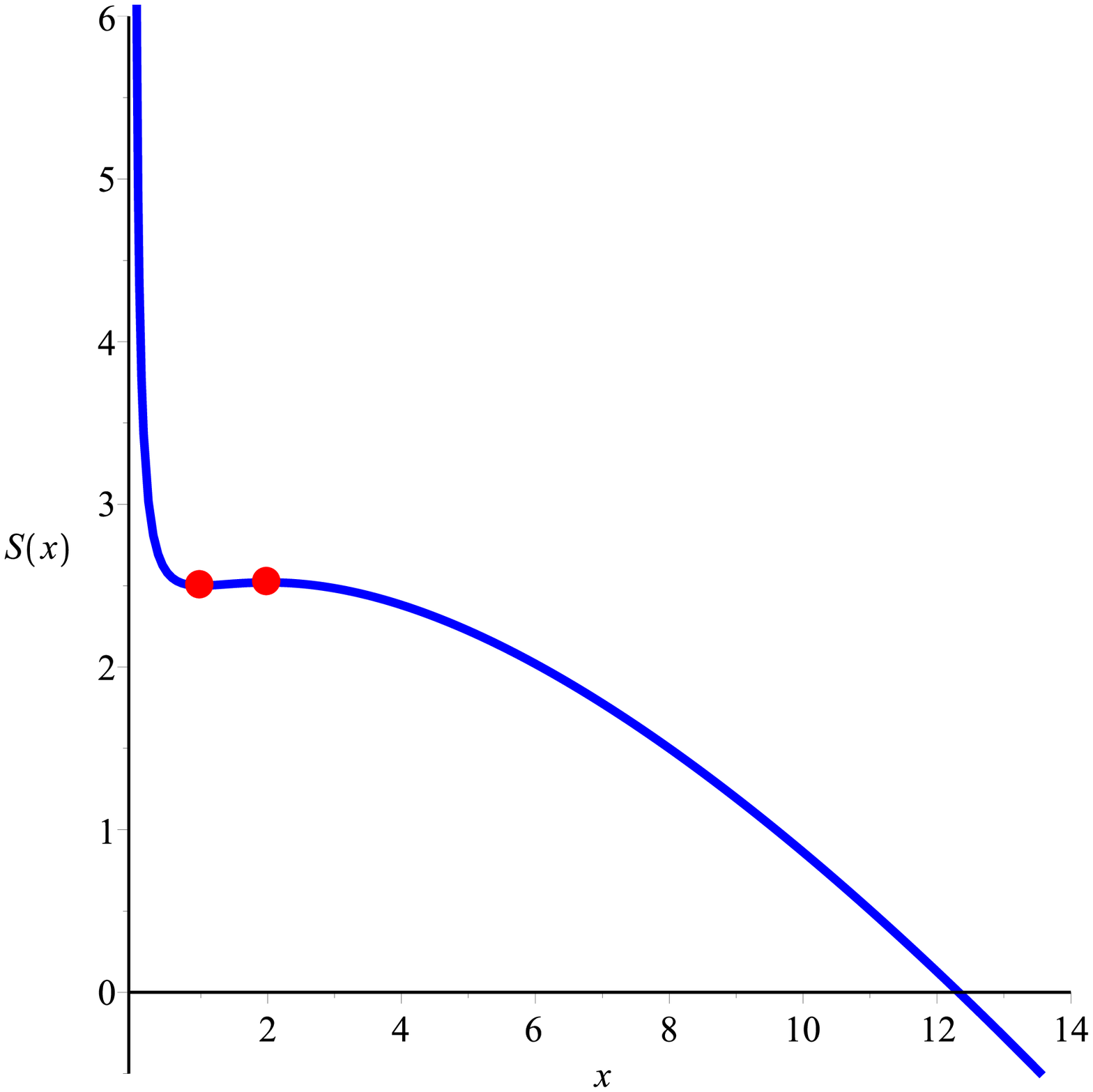}
\includegraphics[angle=0, width=0.45\textwidth]{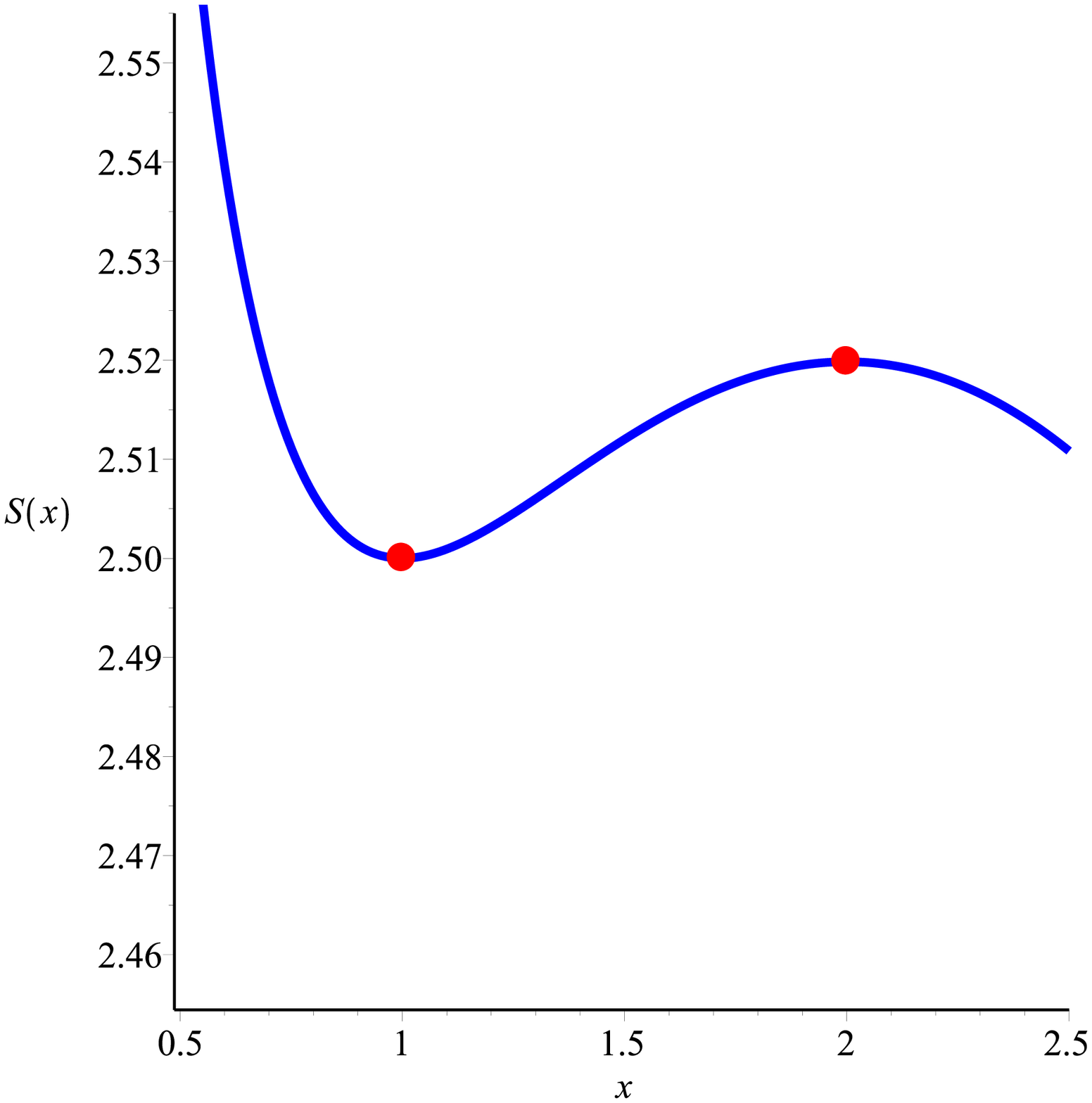}
\caption{The scalar curvature of submersion metrics on the Wallach space $W_6$: the point $x=1$ (the normal metric) is a local minimum,
the point $x=2$ (the K\"{a}hler~--~Einstein metric) is a local maximum.}
\label{OO}
\end{figure*}

We briefly describe  the evolution of submersion metrics under the normalized Ricci flow.
Without loss of generality we may consider the family $(x_1,x_2,x_3)=(x^{-1/3},x^{-1/3},x^{2/3})$, $x\in \mathbb{R}_{+}$, see \eqref{metric}.
It comes from changing the scaling of the fibre
and the base with keeping of the volume.
Here $x$ is the ratio of the multiples of the normal metric on the fibre and on the base.
Since this family is invariant under the Ricci flow, the behavior of the Ricci
flow can be read off the behavior of the scalar curvature function $S(x)$.
The point $x=1$ (the normal metric) is a local minimum and the second Einstein metric (which is the K\"{a}hler~--~Einstein for $W_6$) is a local maximum.
It is well-known, that when starting with the normal homogeneous
metric and shrinking the fibre (i.~e. $x<1$), these metrics will have positive sectional curvature, moreover, $S(x) \to \infty$ as $x \to 0$.
Note also that the non-normal Einstein metric has positive Ricci curvature but mixed sectional curvature.
It is clear also that $S(x)<0$ for sufficiently large $x$. This give us qualitative picture of the Ricci flow's behavior on submersion metrics.
We illustrate this by Figure \ref{OO} for $W_6$, see also Figure~\ref{Fig4z} for a more general context.
Taking into account this description, we  preferably deal  with generic invariant metrics on the Wallach spaces.
More general constructions of the canonical variation for submersion metrics one can find in~\cite[9.72]{Bes}.

\smallskip

In the papers  \cite{AANS1} and \cite{AANS2}, the authors studied the normalized Ricci flow equation
\begin{equation}\label{ricciflow}
\frac {\partial}{\partial t} \bold{g}(t) = -2 {\Ric}_{\bold{g}}+ 2{\bold{g}(t)}\frac{S_{\bold{g}}}{n}
\end{equation}
on one special class of Riemannian manifolds $M^n$
called generalized Wallach spaces (or three-locally-symmetric spaces in other terms) according to the definitions of
\cite{Lomshakov2} and \cite{Nikonorov1},
where $\bold{g}(t)$ means a $1$-parameter family of Riemannian metrics,
$\Ric_{\bold{g}}$ is the Ricci tensor and $S_{\bold{g}}$ is the scalar curvature of the Riemannian metric ${\bold{g}}$.
Generalized Wallach spaces are characterized as compact homogeneous spaces $G/H$ whose isotropy representation
decomposes into a direct sum
$\mathfrak{p}=\mathfrak{p}_1\oplus \mathfrak{p}_2\oplus \mathfrak{p}_3$ of three
$\Ad(H)$-invariant irreducible modules satisfying
$[\mathfrak{p}_i,\mathfrak{p}_i]\subset \mathfrak{h}$ $(i\in\{1,2,3\})$ \cite{Lomshakov2,Nikonorov2}.
The complete classification of generalized Wallach spaces is obtained recently (independently)
in the papers \cite{CKL} and \cite{Nikonorov4}.
For a fixed bi-invariant inner product
$\langle\cdot, \cdot\rangle$ on the Lie algebra $\mathfrak{g}$ of the Lie group $G$,
any $G$-invariant Riemannian metric $\bold{g}$ on $G/H$ is determined by an $\Ad (H)$-invariant inner product
\begin{equation}\label{metric}
(\cdot, \cdot)=\left.x_1\langle\cdot, \cdot\rangle\right|_{\mathfrak{p}_1}+
\left.x_2\langle\cdot, \cdot\rangle\right|_{\mathfrak{p}_2}+
\left.x_3\langle\cdot, \cdot\rangle\right|_{\mathfrak{p}_3},
\end{equation}
where $x_1,x_2,x_3$ are positive real numbers.
Therefore, the space of such metrics is $2$-dimensional up to a scale factor.
Any metric with $x_1=x_2=x_3$ is called {\it normal}, whereas the metric with $x_1=x_2=x_3=1$ is called {\it standard} or {\it Killing}.
Metrics with pairwise distinct $x_i$, $i=1,2,3$, we call generic as in the case of the Wallach spaces.
\smallskip

The Ricci curvature of the metric (\ref{metric}) could be easily expressed in terms of special constants $a_1$, $a_2$, and $a_3$,
that determine a given generalized Wallach space, see details e.~g. in \cite{AANS1}.
Note that $a_1=a_2=a_3=:a$ and $\dim({\mathfrak{p}_1})=\dim({\mathfrak{p}_2})=\dim({\mathfrak{p}_3})=:{\bf d}$
for the Wallach spaces $W_6$, $W_{12}$, and $W_{24}$. Moreover, for these spaces,
$a$ is equal to $1/6$, $1/8$, $1/9$ and ${\bf d}$ is equal to $2$, $4$, $8$ respectively.

It should be noted that Theorem \ref{Sect_Ricci} can be extended to some other generalized Wallach spaces.

\begin{theorem}\label{Sect_Ricci_gen}
Let $G/H$ be a generalized Wallach space with $a_1=a_2=a_3=:a$, where $a\in (0,1/4)\cup(1/4,1/2)$.
If $a<1/6$, then
the normalized Ricci flow
evolves all generic metrics  with positive Ricci curvature  into metrics with mixed Ricci curvature.
If $a\in (1/6,1/4)\cup(1/4,1/2)$, then the normalized Ricci flow evolves all generic metrics into metrics with positive Ricci curvature.
\end{theorem}

For instance, the spaces $Sp(3k)/Sp(k)\times Sp(k)\times Sp(k)$ correspond to the case $a=\frac{k}{6k+2} <1/6$,
whereas the spaces $SO(3k)/SO(k)\times SO(k)\times SO(k)$, $k>2$, correspond to the case $1/6<a=\frac{k}{6k-4}<1/4$.
Note also that $SO(6)/SO(2)\times SO(2)\times SO(2)$ corresponds to $a=1/4$, that is a very special case of generalized Wallach spaces
with a unique Einstein invariant metric up to a homothety,
and $SO(3)$ correspond to $a=1/2$, the maximal possible value for $a=a_1=a_2=a_3$, see details
in \cite{AANS1} and \cite{AANS2}.
It is interesting also that $1/9$ is the minimal possible value for $a=a_1=a_2=a_3$ among non-symmetric generalized Wallach spaces, see \cite{Nikonorov4}.

It should also be noted that there are many generalized Wallach spaces with $a=1/6$, for example,
the spaces $SU(3k)/S(U(k)\times U(k)\times U(k))$. All these spaces are K\"{a}hler C-spaces, see \cite{Nikonorov4}. We state the following result,
that generalizes Theorem 8 of~\cite{ChWal}.

\begin{theorem}\label{Sect_Ricci_genn}
Let $G/H$ be a generalized Wallach space with $a_1=a_2=a_3=1/6$.
Suppose that it is supplied with the invariant Riemannian metric \eqref{metric} such that $x_k<x_i+x_j$ for all indices with $\{i,j,k\}=\{1,2,3\}$,
then the normalized Ricci flow on $G/H$ with this metric as the initial point, preserves the positivity of the Ricci curvature.
\end{theorem}

It should be noted that $x_k=x_i+x_j$ is just the unstable manifold of the K\"{a}hler~--~Einstein metric for all generalized Wallach spaces with $a=1/6$.
\smallskip

Note that our results correlate with the results of the papers \cite{Bo} and~\cite{ChWal},
but our approach is mainly based on a more detailed study of the asymptotic behavior of integral curves  of the normalized Ricci flow for $t \to \infty$,
see Proposition \ref{asymp3}.
Another important ingredient is an useful and detailed description of metrics with positive sectional and positive Ricci curvature.
The set of metrics with positive sectional curvature are described in details in Section \ref{sectionalm}, which is based on the original paper \cite{Valiev}
of F.\,M.~Valiev.
The comprehensive description of the set of metrics with positive Ricci curvature is given in Section~\ref{on_R_set}.
We hope that our illustrations help to imagine and ``feel'' these important sets of metrics.

\smallskip

The paper is organized as follows:
In Section \ref{ricci computed} we reduce the normalized Ricci flow equation (\ref{ricciflow})  to the system of ODE's
 (\ref{rnrf_scr}) and get some important properties of solutions of this system.
In Section \ref{sectionalm} we study the evolution of metrics with positive sectional curvature and prove Theorem \ref{the2}.
The next section, where Theorem \ref{Sect_Ricci}, Theorem \ref{Sect_Ricci_gen}, and Theorem \ref{Sect_Ricci_genn} are proved,
is devoted to the evolution of metrics with positive Ricci curvature.
In the final section we briefly discuss  the evolution of invariant metrics with positive scalar curvature and give additional illustrations of the behavior
of normalized Ricci flow on the Wallach spaces.

\section{Reduction of the normalized Ricci flow equation to a system of ODE's}\label{ricci computed}

Recall that $\dim({\mathfrak{p}_1})=\dim({\mathfrak{p}_2})=\dim({\mathfrak{p}_3})=:{\bf d}$ for the Wallach spaces $W_6$, $W_{12}$, and $W_{24}$,
where ${\bf d}=2$, ${\bf d}=4$, and ${\bf d}=8$ respectively.
As noted above, the Ricci curvature of the metric (\ref{metric}) for these spaces could be easily expressed in the term of a special constant $a$,
that is equal to $1/6$, $1/8$, and $1/9$ respectively, see details e.~g. in \cite{AANS1}.
Note that in  \cite{ChWal}, the Wallach spaces \eqref{SWS} are characterized by the values of ${\bf d}$, which is
connected with our $a$ by the relation $a=\frac{{\bf d}}{10{\bf d}-8}$.
Note also that the formulae below are valid also for all generalized Wallach spaces with the property
$\dim({\mathfrak{p}_1})=\dim({\mathfrak{p}_2})=\dim({\mathfrak{p}_3})=:{\bf d}$, which is equivalent to $a_1=a_2=a_3=:a$
(see \cite{AANS1}). We will consider  such spaces only for  $a\in (0,1/4)\cup (1/4,1/2)$, because every generalized Wallach space with $a=1/4$
admits a unique Einstein metric up to a homothety, see \cite{AANS1} for detailed discussion.
\smallskip

Recall that the Ricci operator $\Ric$ of the metric (\ref{metric}) is given by
$$
\Ric=\left.{\bf r_1}\, \Id \right|_{\mathfrak{p}_1}+
\left.{\bf r_2}\, \Id \right|_{\mathfrak{p}_2}+
\left.{\bf r_3}\, \Id \right|_{\mathfrak{p}_3},
$$
where
$${\bf r_i}:=\frac{x_jx_k+a(x_i^2-x_j^2-x_k^2)}{2x_1x_2x_3}$$
 are the principal Ricci curvatures,
$\{i,j,k\}=\{1,2,3\}$.
Hence, the scalar curvature of this metric is
$$
S={\bf d}\cdot \frac{x_1x_2+x_1x_3+x_2x_3-a(x_1^2+x_2^2+x_3^2)}{2x_1x_2x_3}.
$$

\smallskip

By using the above equalities,
the (volume) normalized Ricci flow equation~(\ref{ricciflow}) on the Wallach spaces can be
reduced to a system of ODE's of the following form:
\begin{equation}\label{nrf_sc}
\frac {dx_i}{dt} = -2x_i(t)\left({\bf r_i}-\frac{S}{n}\right), \quad i=1,2,3,
\end{equation}
where
$n=\dim({\mathfrak{p}_1})+\dim({\mathfrak{p}_2})+\dim({\mathfrak{p}_3})=3{\bf d}$.

Note that the metric (\ref{metric}) has the same volume as the standard metric if and only if
$x_1x_2x_3=1$.
It suffices to prove  Theorems \ref{the2},  \ref{Sect_Ricci}, \ref{Sect_Ricci_gen}, and \ref{Sect_Ricci_genn}
only for invariant metrics with
\begin{equation}\label{volume}
\Vol=x_1x_2x_3\equiv 1.
\end{equation}
Indeed, the case of general volume is reduced to this one by a suitable homothety.
This observation is the main argument to apply the normalized Ricci flow instead of the non-normalized Ricci flow
in the case of the Wallach spaces, as far as in the case of generalized Wallach spaces, see details in \cite{AANS1} and \cite{AANS2}.

It is easy to check that $\Vol=x_1x_2x_3$ from~(\ref{volume}) is a first integral of the system~(\ref{nrf_sc}).
Therefore,
we can reduce (\ref{nrf_sc}) to the following system of two differential equations on the surface (\ref{volume}):

{\renewcommand{\arraystretch}{1.6}
\begin{equation}\label{rnrf_sc}
\begin{array}{c}
\frac {dx_1}{dt} = \left(x_1x_2^{-1}+x_1^2x_2-2 \right)- 2ax_1\left(2x_1^2-x_2^2-x_1^{-2}x_2^{-2}\right),      \\
\frac {dx_2}{dt} = \left(x_2x_1^{-1}+x_1x_2^2-2 \right)- 2ax_2\left(2x_2^2-x_1^2-x_1^{-2}x_2^{-2}\right).
\end{array}
\end{equation}
}

For our goals we need also a system of ODE's obtaining in scale invariant variables
\begin{equation}\label{siv}
w_1:=\frac{x_3}{x_1}, \quad w_2:=\frac{x_3}{x_2}.
\end{equation}

Since (\ref{nrf_sc}) is autonomous and
$$\frac{1}{w_i}\frac{dw_i}{dt}=\frac{1}{x_3}\frac{dx_3}{dt}-\frac{1}{x_i}\frac{dx_i}{dt}=-2({\bf r_3}-{\bf r_i}),$$
for $i=1,2$,  then (\ref{nrf_sc}) can be reduced to the following system for $w_1>0$ and $w_2>0$:
{\renewcommand{\arraystretch}{1.6}
\begin{equation}\label{rnrf_scr}
\begin{array}{c}
\frac {dw_1}{dt} = f(w_1,w_2):= (w_1-1)(w_1-2aw_1w_2-2aw_2),      \\
\frac {dw_2}{dt} = g(w_1,w_2):=(w_2-1)(w_2-2aw_1w_2-2aw_1),
\end{array}
\end{equation}
}
where   $t:=tx_3$ is a  new time-parameter
not changing integral curves and their orientation ($x_3>0$).

\medskip

In the first version of this paper (see \cite{AN}), we used the system (\ref{rnrf_sc}) as the main tool, but here we deal  with  (\ref{rnrf_scr}) preferably
according to the referee advice.
Comparisons show  that the system (\ref{rnrf_scr}) in the scale invariant variables $(w_1,w_2)$ is more convenient
in order to prove our main theorems. On the other hand, we prefer to give
visual interpretations of the results in both  coordinate systems $(w_1,w_2)$ and $(x_1,x_2)$.
Further we will follow this strategy.

\subsection{Singular points and invariant curves of  the system (\ref{rnrf_scr})}
The following  lemma can be proved by simple and direct calculations
(see the left panel of Figure~\ref{Fig1}).

\begin{lemma}\label{E_i_hyperr}
Let $w_1>0$ and $w_2>0$. Then
\begin{enumerate}
\item
The  curves $c_1,c_2$ and $c_3$
determined by the equations
$$
w_2=1, \quad w_1=1 \quad \mbox{and}\quad w_2=w_1
$$
respectively
are invariant sets of the system  {\rm(}\ref{rnrf_scr}{\rm)};
\item
At  $a\ne 1/4$, the system {\rm(}\ref{rnrf_scr}{\rm)} has exactly four different singular points
$E_0=(1,1)$, \, $E_1=(q,1)$, \,  $E_2=(1,q)$, \,  $E_3=\left(q^{-1}, q^{-1}\right)$,
where
\begin{equation}\label{ququr}
q:=2a(1-2a)^{-1}.
\end{equation}
Moreover,  $E_1, E_2$ and $E_3$ are hyperbolic saddles and $E_0$ is a hyperbolic unstable node.
\end{enumerate}
\end{lemma}

\medskip

Note that the curves $c_1,c_2$ and $c_3$ have the common point $E_0$  and
separate the domain $(0,\infty)^2$ into $6$  connected invariant components (see the left panel of Figure~\ref{Fig1}).
The study of normalized Ricci flow in each pair of these components are equivalent due to the following property of the Wallach spaces:
there is a finite group of isometries fixing the isotropy and permuting the modules $\mathfrak{p}_1$, $\mathfrak{p}_2$, and $\mathfrak{p}_3$.
Therefore, it suffices to
study solutions of (\ref{rnrf_scr}) with initial points given only in the following set
\begin{equation}\label{Omegar}
\Omega:=\left\{(w_1,w_2)\in \mathbb{R}^2~|~ w_2>w_1>1\right\}.
\end{equation}

\begin{figure*}[t]
\centering
\includegraphics[angle=0, width=0.45\textwidth]{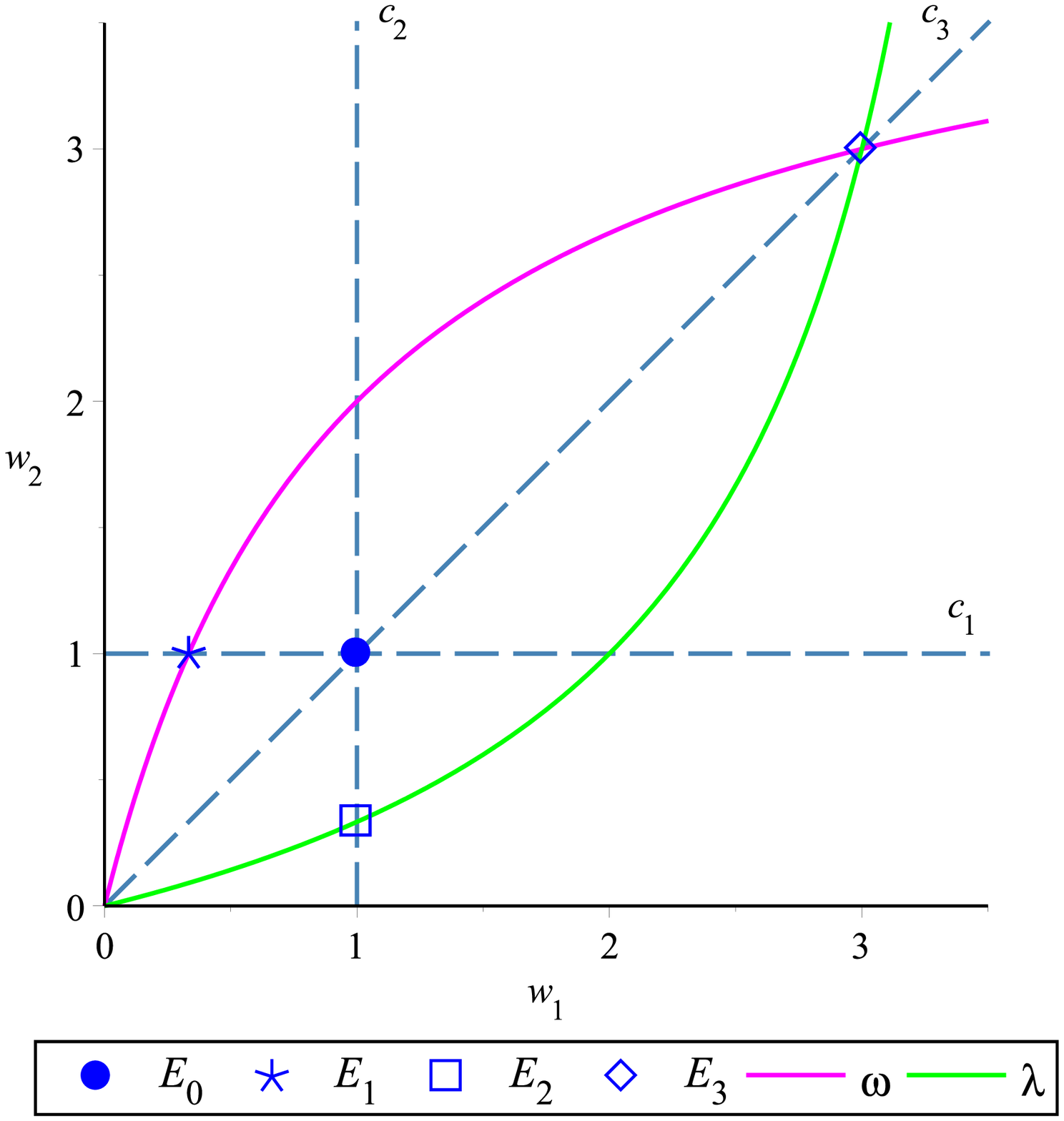}
\includegraphics[angle=0, width=0.45\textwidth]{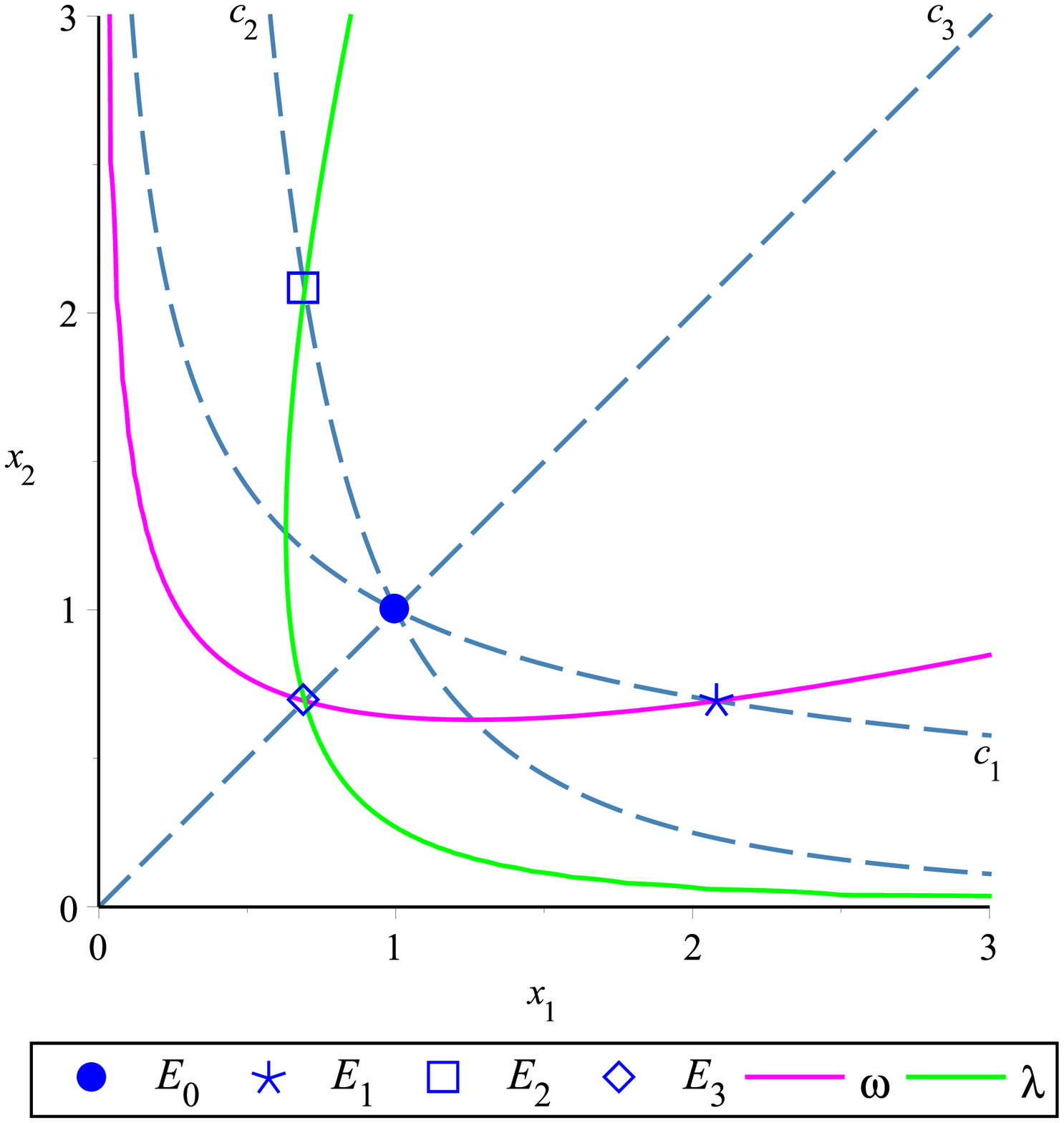}
\caption{The curves $c_1,c_2,c_3$ and the singular points $E_0,E_1,E_2,E_3$
corresponding to the systems (\ref{rnrf_scr}) (the left panel) and (\ref{rnrf_sc}) (the right panel)
for $a=1/8$.}
\label{Fig1}
\end{figure*}

A simple analysis of the right hand sides of the system (\ref{rnrf_scr})
provides elementary tools for studying the behavior of its integral curves.
For instance, we can predict the slope  of integral curves of (\ref{rnrf_scr}) in $\Omega$
and interpret them geometrically.
According to this observations, let us consider
the sets (see the left panel of Figure~\ref{Fig1})
\begin{eqnarray*}\label{Omegapr}
\Omega'&:=&\left\{(w_1,w_2)\in \Omega\,\,\,\,~|~ w_1-2aw_1w_2-2aw_2>0\right\},\\
\omega&:=&\left\{(w_1,w_2)\in \mathbb{R}_{+}^2~|~ w_1-2aw_1w_2-2aw_2=0\right\},\\
\lambda&:=&\left\{(w_1,w_2)\in \mathbb{R}_{+}^2~|~ w_2-2aw_1w_2-2aw_1=0\right\}.
\end{eqnarray*}

Denote by $\overline{\Omega'}$ the closure of $\Omega'$ in the standard topology of $\mathbb{R}^2$.

Let  $(w_1(t), w_2(t))$ be any integral curve of  (\ref{rnrf_scr})
given in $\Omega$.
Then $w_1'(t)>0$ and $w_2'(t)>0$ in $\Omega'$ ({\it under $\omega$}).
In the set $\Omega\setminus \overline{\Omega'}$  we have the following:
$w_1'(t)<0$ and  $w_2'(t)>0$  over  $\omega\cup \lambda$;
$w_1'(t)<0$ and  $w_2'(t)<0$  under  $\lambda$.
Clearly,  $w_1'(t)=0$ on $\omega$ and $w_2'(t)=0$ on $\lambda$.

Note also that the curves $\omega$ and $\lambda$ consist of invariant metrics with the equality ${\bf r_3}={\bf r_1}$
and ${\bf r_3}={\bf r_2}$ for the principal Ricci curvatures.

\subsection{Asymptotic behavior of solutions of the system (\ref{rnrf_scr})}

Consider an arbitrary
trajectory of the system (\ref{rnrf_scr}) with
initial point given in $\Omega$.
Then clearly, $\lim\limits_{t \to +\infty}w_2(t)=+\infty$
by influence of the stable and unstable manifolds of the saddle $E_3$.
Therefore, there is $\hat{t}$ such that $(w_1(t),w_2(t))$ leaves the compact
$\overline{\Omega'}$ for $t>\hat{t}$.
It is also clear from (\ref{rnrf_scr}) that  $w_1(t)\to 1+0$ as $t \to +\infty$
since $w_1'(t)<0$ for $t>\hat{t}$.
Note also that $\lim\limits_{t\to -\infty}w_1(t)=\lim\limits_{t\to -\infty}w_2(t)=1$
for any trajectory of (\ref{rnrf_scr}), which passes through a point of the set $\overline{\Omega'}$.
Hence,  the  behavior of integral curves of (\ref{rnrf_scr}) in the set $\overline{\Omega'}$ is clear.

\smallskip

Now we should study integral curves of  (\ref{rnrf_scr})  in $\Omega\setminus \overline{\Omega'}$
estimating their  ``curvature'' as $w_1 \to 1+0$ and $w_2\to +\infty$ more precisely.
For this goal observe that at $(w_1,w_2)\in \Omega\setminus \overline{\Omega'}$ the system (\ref{rnrf_scr})
is equivalent to the equation

\begin{equation}\label{ctoto}
\frac{dw_2}{dw_1}= \frac{g}{f}=\frac{(w_2-1)(w_2-2aw_1w_2-2aw_1)}{(w_1-1)(w_1-2aw_1w_2-2aw_2)},
\end{equation}
where
$$
\frac{g}{f} \to -\infty \quad \mbox{as} \quad w_1\to 1+0 \quad \mbox{and}\quad  w_2\to +\infty.
$$

Let $w_2:=\phi(w_1)$ be a solution of (\ref{ctoto}).
In fact we are going to reformulate the question above (about ``curvature'') as the problem of detecting the asymptotic  behavior of
$w_2:=\phi(w_1)$ when $w_1 \to 1+0$.

\begin{lemma}\label{prikol}
Let $w_2=\phi(w_1)$ be a solution of {\rm(}\ref{ctoto}{\rm)}, where $(w_1,w_2)\in \Omega\setminus \overline{\Omega'}$.
Then for any small $\varepsilon>0$ there exist  constants $C_1,C_2>0$ such that
$$
C_1(w_1-1)^{-\frac{(1-\varepsilon)(1-2a)}{4a}}\leq \phi(w_1) \leq C_2(w_1-1)^{-\frac{(1+\varepsilon)(1-2a)}{4a}}
$$
for $w_1$ sufficiently close to $1$ and $w_1>1$.
\end{lemma}

\begin{proof}
An easy analysis shows that
$$
\frac{g}{f} \sim \frac{2a-1}{4a}\frac{w_2}{w_1-1}
$$
as $w_1\to 1+0$.
Therefore,
$
1-\varepsilon\leq \frac{4a}{2a-1}\frac{w_1-1}{w_2} \frac{dw_2}{dw_1} \leq 1+\varepsilon
$
for a sufficiently small $\varepsilon>0$ and $w_1$ close to $1$.
Taking $w_1'$ and $w_1''$  (assuming $1<w_1'<w_1''$)  close to~$1$, we get
$$
(1-\varepsilon)\int\limits_{w_1'}^{w_1''}\frac{dw_1}{w_1-1} \leq \frac{4a}{2a-1}
\int\limits_{\phi (w_1')}^{\phi (w_1'')} \frac{dw_2}{w_2}\leq
(1+\varepsilon)\int\limits_{w_1'}^{w_1''}\frac{dw_1}{w_1-1}
$$
which is equivalent to
$$
\left(\frac{w_1''-1}{w_1'-1}\right)^{1-\varepsilon} \leq
\left(\frac{w_2''}{w_2'}\right)^{\frac{4a}{2a-1}} \leq
\left(\frac{w_1''-1}{w_1'-1}\right)^{1+\varepsilon}.
$$
This means that for any small $\varepsilon>0$ there exist constants $C_1,C_2>0$ such that
$$
C_1(w_1-1)^{-(1-\varepsilon)}\leq
w_2^{\frac{4a}{1-2a}}
\leq
C_2(w_1-1)^{-(1+\varepsilon)}
$$
for $w_1$ sufficiently close to $1$
(at fixed $w_1''$ and $w_1':=w_1$).
\end{proof}

\begin{proposition}\label{asymp3}
Suppose that a curve $\gamma$ given in $\Omega$ satisfies the asymptotic equality
$$
w_2\sim (w_1-1)^{-\alpha} \quad \mbox{as}\quad  w_1 \to 1+0,
$$
where $\alpha>0$.
Then the following assertion holds:
If $\frac{1-2a}{4a}<\alpha$ {\rm(}respectively, $\frac{1-2a}{4a}>\alpha${\rm)}, then every integral curve
$(w_1(t),w_2(t))$ of {\rm(}\ref{rnrf_scr}{\rm)} in $\Omega$ lies under {\rm(}respectively, over{\rm)}  $\gamma$ for sufficiently large $t$.
\end{proposition}

\begin{proof}
Recall that $w_1 \to 1+0$ and $w_2\to +\infty$ as $t \to +\infty$ on every integral curve  of (\ref{rnrf_scr}).
In Lemma \ref{prikol} we may take $\varepsilon>0$ such that $\varepsilon<\left|1-\frac{4a\alpha}{1-2a}\right|$.
If $\frac{1-2a}{4a}<\alpha$, then $\frac{(1+\varepsilon)(1-2a)}{4a}<\alpha$. This means that
$$
\lim\limits_{w_1\to 1+0} \frac{\phi(w_1)}{(w_1-1)^{-\alpha}}\leq C_2
\lim\limits_{w_1\to 1+0}\frac{(w_1-1)^{-\frac{(1+\varepsilon)(1-2a)}{4a}}}{(w_1-1)^{-\alpha}}=0
$$
and the integral curve lies  {\it under} the curve $\gamma$  for all $w_1$ sufficiently close to $1$ and $w_1>1$.

If $\frac{1-2a}{4a}>\alpha$, then $\frac{(1-\varepsilon)(1-2a)}{4a}>\alpha$. This means that
$$
+\infty=C_1
\lim\limits_{w_1 \to 1+0}\frac{(w_1-1)^{-\frac{(1-\varepsilon)(1-2a)}{4a}}}{(w_1-1)^{-\alpha}}\leq
\lim\limits_{w_1 \to 1+0}\frac{\phi(w_1)}{(w_1-1)^{-\alpha}}
$$
and the integral curve  lies {\it over}  the curve $\gamma$  for all $w_1$ sufficiently close to $1$ and $w_1>1$.
\end{proof}

\begin{figure*}[t]
\centering
\includegraphics[angle=0, width=0.45\textwidth]{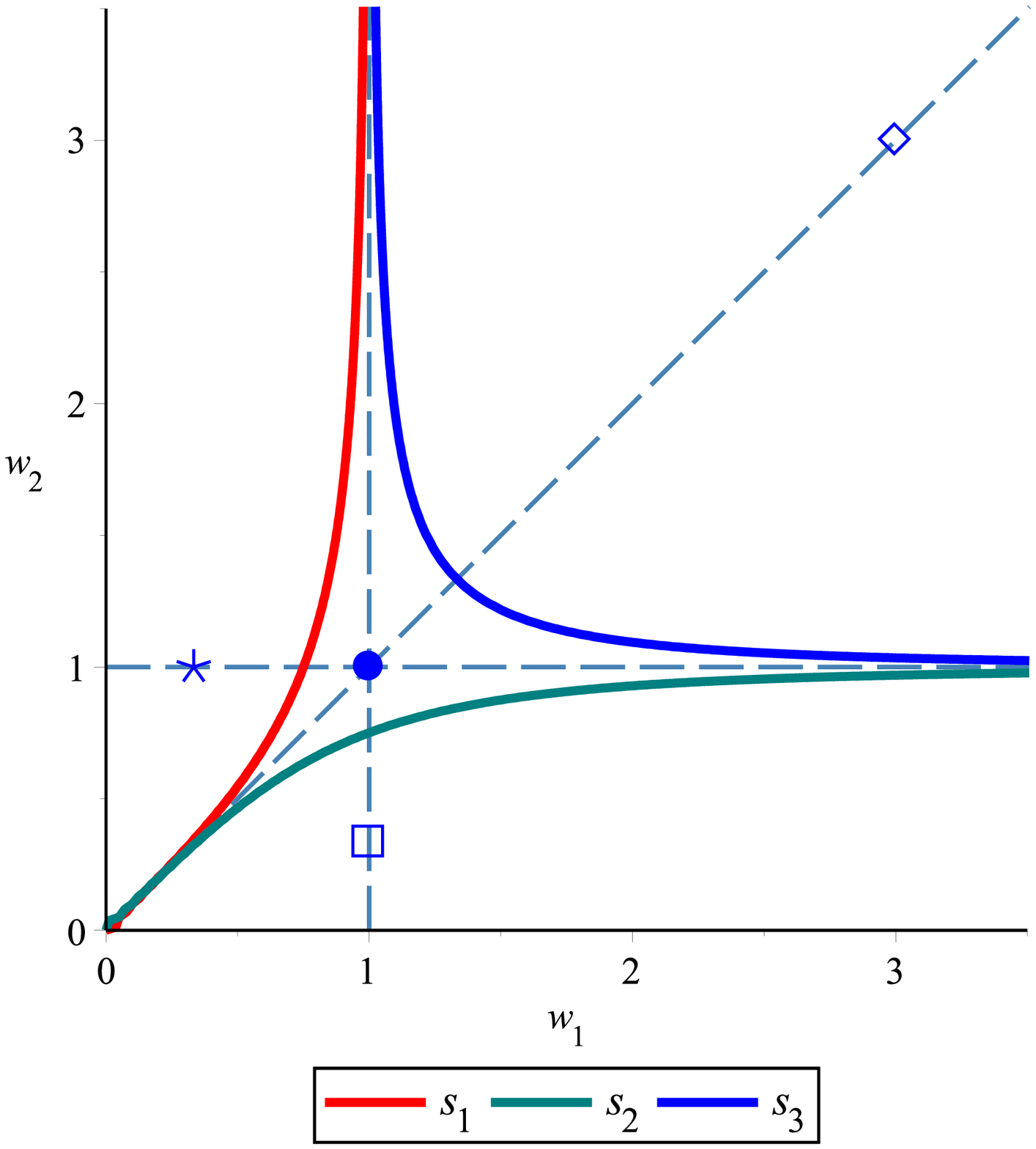}
\includegraphics[angle=0, width=0.45\textwidth]{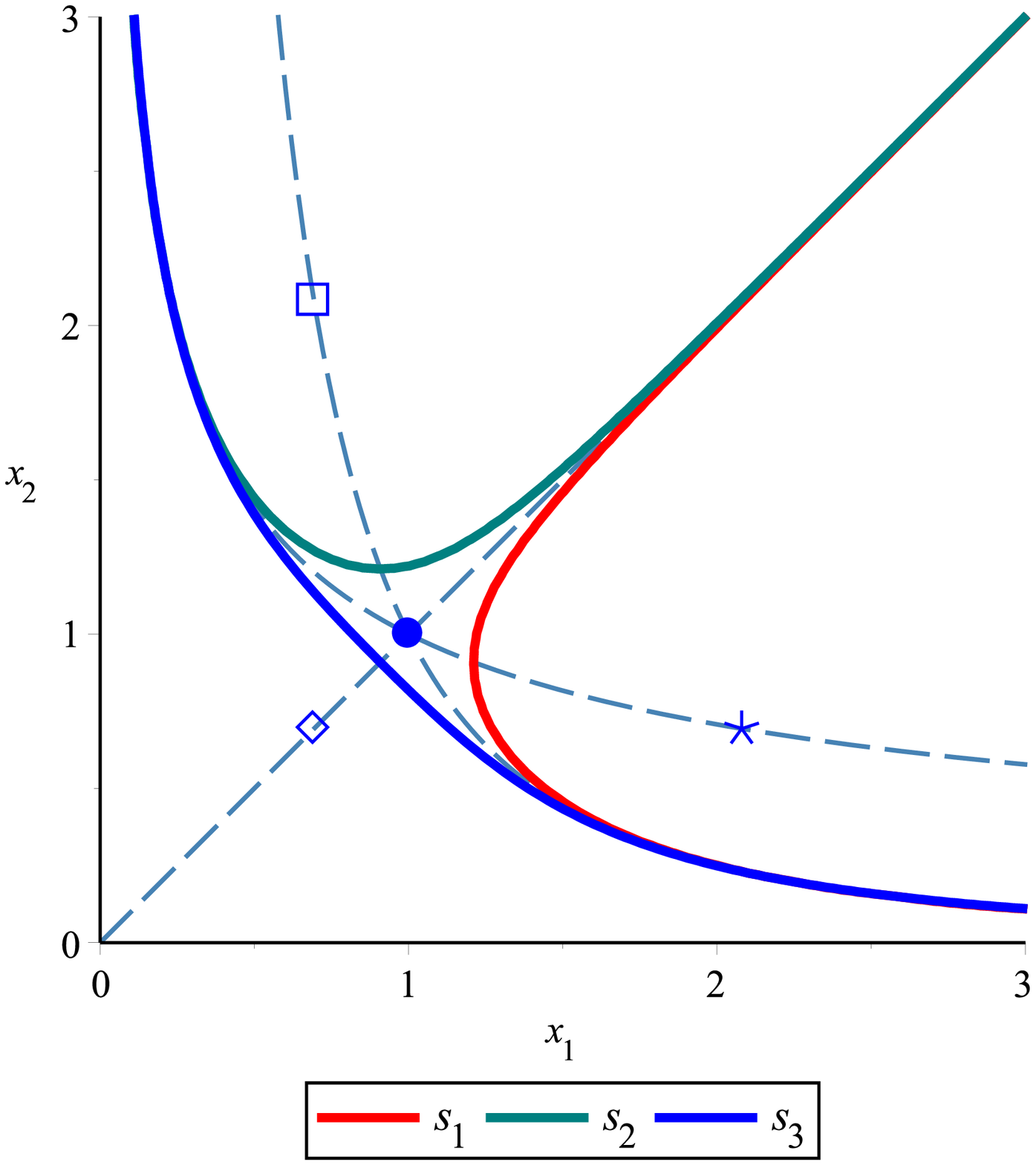}
\caption{The curves $s_1,s_2,s_3$ corresponding to the systems (\ref{rnrf_scr})
(the left panel) and (\ref{rnrf_sc}) (the right panel).}
\label{Fig2a}
\end{figure*}

\section{Evolution of invariant  metrics with positive sectional curvature}\label{sectionalm}

A detailed description of invariant metrics of positive sectional curvature on the Wallach spaces (\ref{SWS})
was given  by F.\,M.~Valiev in  \cite{Valiev}. We reformulate his results in our notation.
Let us fix a Wallach space $G/H$ (i.~e. consider $a=1/6$, $a=1/8$, or $a=1/9$).

Recall that we deal with only positive $x_i$. Let us consider the functions
$$
\gamma_i=\gamma_i(x_1,x_2,x_3):=(x_j-x_k)^2+2x_i(x_j+x_k)-3x_i^2,
$$
where
$\{i,j,k\}=\{1,2,3\}$.
Note that under the restrictions $x_i>0$, the equations
$\gamma_i=0$, $i=1,2,3$, determine
cones congruent each to other  under the permutation $i\rightarrow j \rightarrow k \rightarrow i$.
Note also that these cones have the empty intersections pairwise.

\smallskip

According to results of  \cite{Valiev} and the symmetry in $\gamma_1, \gamma_2$, and $\gamma_3$ under permutations of $x_1$, $x_2$, and $x_3$,
the set of metrics with non-negative sectional curvature is the following:
\begin{equation}\label{DNNG}
\left\{(x_1,x_2,x_3)\in \mathbb{R}_{+}^3 \,\, | \,\, \gamma_1 \geq 0,\, \gamma_2\geq 0,\, \gamma_3\geq 0 \right\}.
\end{equation}

By Theorem~3 in \cite{Valiev} and the above mentioned symmetry,
the set of metrics with positive sectional curvature is the following:
\begin{equation}\label{DPG}
\left\{(x_1,x_2,x_3)\in \mathbb{R}_{+}^3 \,\, | \,\, \gamma_1>0,\, \gamma_2>0,\, \gamma_3>0 \right\}\setminus  \left\{(t,t,t)\in \mathbb{R}^3 \,\, | \,\, t>0 \right\}.
\end{equation}

Let us describe the domain $D$ in the coordinates $(w_1,w_2)$.
Denote by  $s_i$  curves  on the plane $(w_1,w_2)$ determined
by the equations $\gamma_i\bigl(\frac{1}{w_1},\frac{1}{w_2},1\bigr)=0$
 (see the left panel of Figure~\ref{Fig2a}). For $w_1>0$ and $w_2>0$,
 these equations are respectively equivalent to
{\renewcommand{\arraystretch}{1.5}
\begin{equation}\label{L_ir}
\begin{array}{l}
l_1:=\,\,\,\,w_1^2w_2^2-2w_1^2w_2+2w_1w_2^2+w_1^2+2w_1w_2-3w_2^2=0,\\
l_2:=\,\,\,\,w_1^2w_2^2+2w_1^2w_2-2w_1w_2^2-3w_1^2+2w_1w_2+w_2^2=0,\\
l_3:=-3w_1^2w_2^2+2w_1^2w_2+2w_1w_2^2+w_1^2-2w_1w_2+w_2^2=0.
\end{array}
\end{equation}}

It is easy to check that  (\ref{DNNG}) is a connected set
with a boundary consisting of the union of the cones $\gamma_1=0, \gamma_2=0$ and $\gamma_3=0$.
Therefore, solving the system of inequalities $\gamma_i\bigl(\frac{1}{w_1},\frac{1}{w_2},1\bigr)>0$, $i=1,2,3$,
we get a connected domain on the plane $(w_1,w_2)$ bounded by the  curves
 $s_1,s_2$ and $s_3$.  Let us denote it  by~$D$.
We also observe that
$s_i\cap s_j=\emptyset$ for $w_1>0$ and $w_2>0$, where $i\ne j$.

\medskip
It is clear that the only singular point of the systems {\rm(}\ref{rnrf_sc}{\rm)} or {\rm(}\ref{rnrf_scr}{\rm)} that belongs
to the domain $D$ is the unstable node $(1,1)$
in the both coordinate systems $(w_1,w_2)$ and $(x_1,x_2)$.

\begin{remark}\label{D_rema}
Taking into account homotheties, it suffices  to prove Theorem \ref{the2}
for invariant metrics $\bigl(\frac{1}{w_1},\frac{1}{w_2},1\bigr)$ with $(w_1,w_2)\in D\setminus\{(1,1)\}$
in the coordinates $(w_1,w_2)$.
\end{remark}

In what follows, we will need the curves $c_1$, $c_2$ and $c_3$ introduced in Lemma~\ref{E_i_hyperr}.

\begin{lemma}\label{attains}
If $a\in (0,1/4)$ then every  trajectory of the system {\rm(}\ref{rnrf_scr}{\rm)} originated in
$D\setminus \left(c_1\cup c_2\cup c_3\right)$ reaches the boundary $s_1\cup s_2\cup s_3$ of $D$ in finite time
and leaves $D$.
This finite time could be as long as we want.
\end{lemma}

The corresponding picture
is depicted in the left panel of Figure \ref{Fig5a}.

\begin{figure*}[t]
\centering
\includegraphics[angle=0, width=0.45\textwidth]{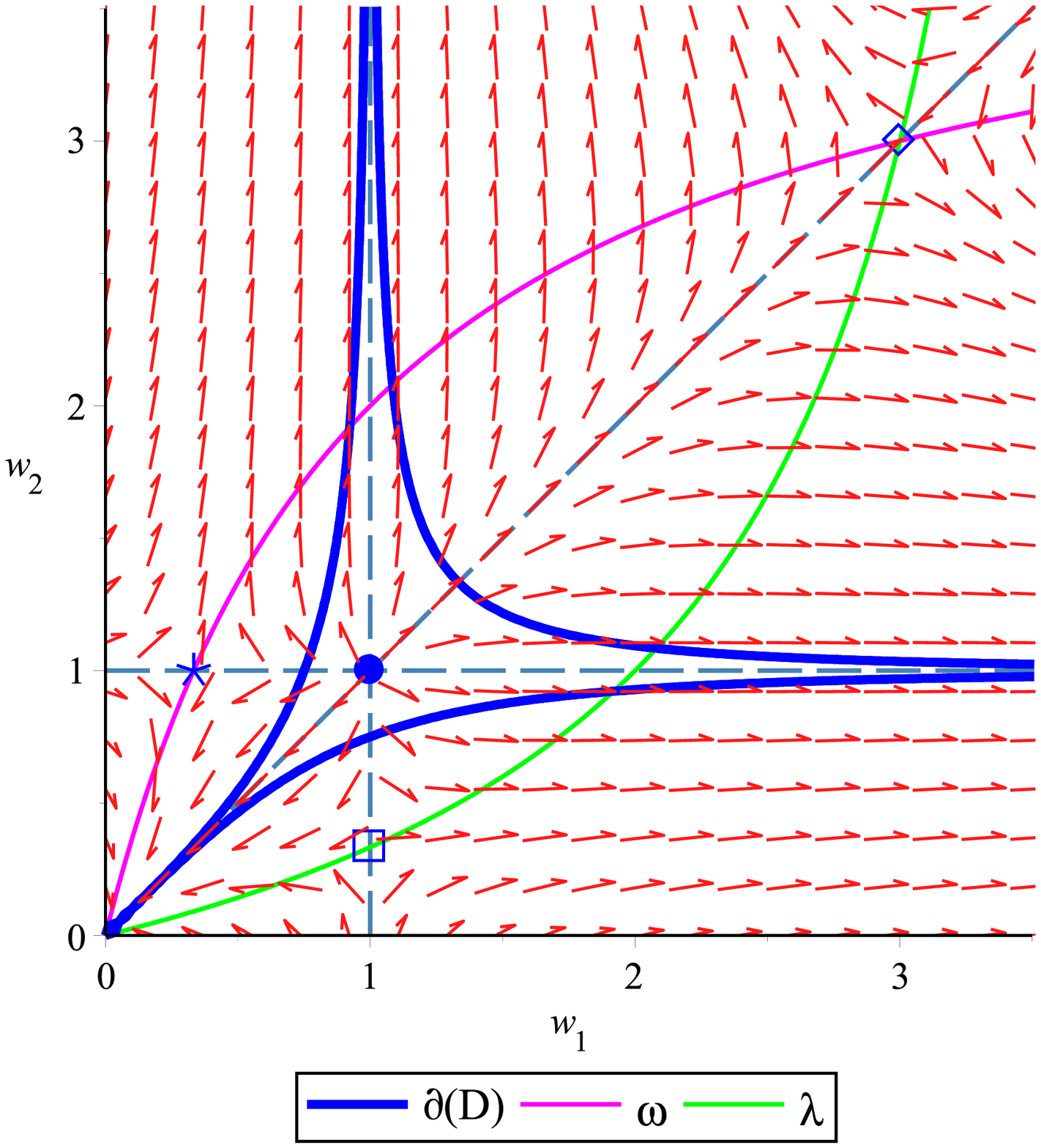}
\includegraphics[angle=0, width=0.45\textwidth]{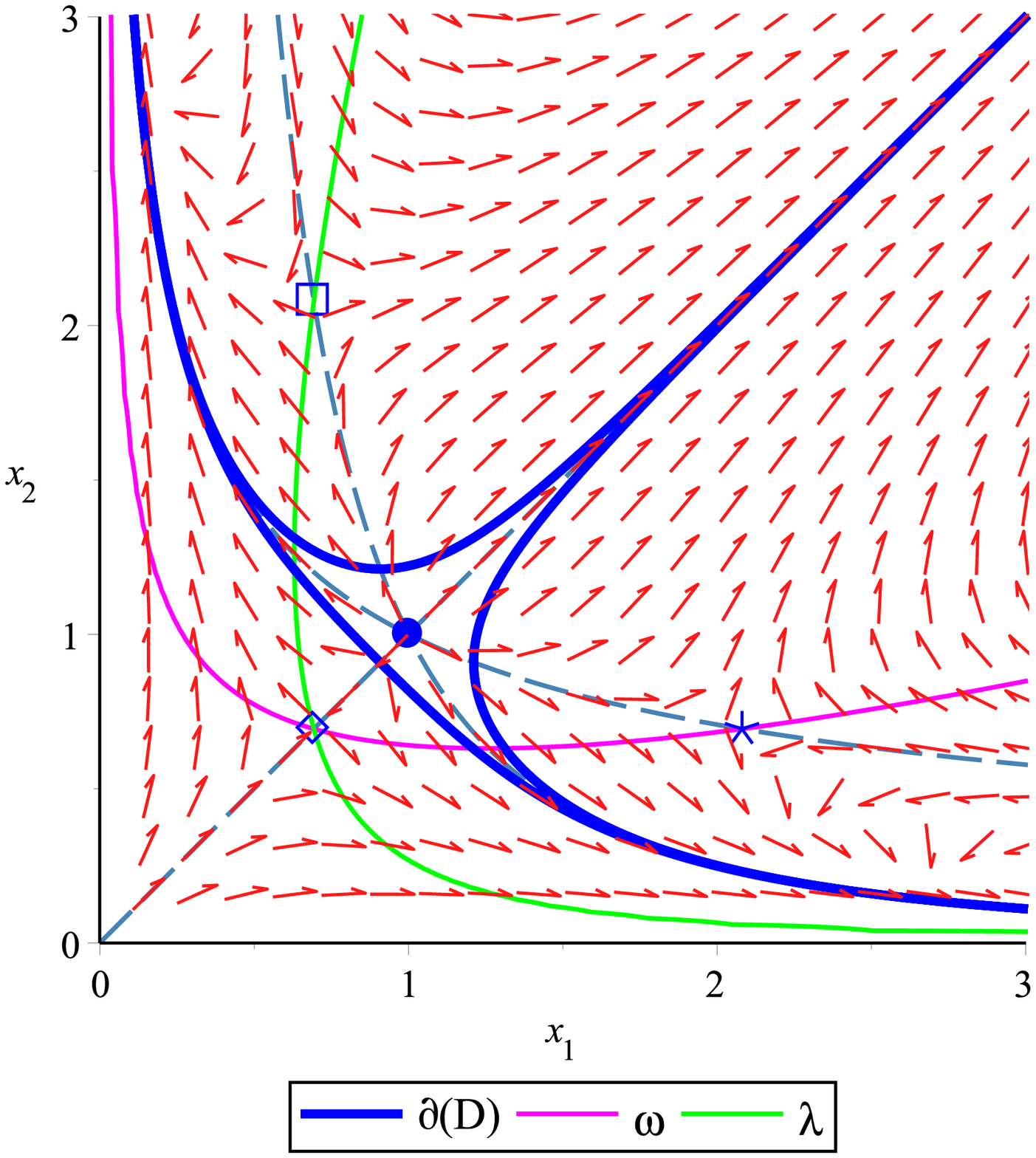}
\caption{The case $a=1/8$: The phase portraits of the systems (\ref{rnrf_scr}) (the left panel) and (\ref{rnrf_sc}) (the right panel).}
\label{Fig5a}
\end{figure*}

\begin{proof}
Without loss of generality  consider only the part $D\cap \Omega$ of $D$,
where $\Omega$ was introduced in (\ref{Omegar}).
Consider any trajectory $\bigl(w_1(t), w_2(t)\bigr)$
of~(\ref{rnrf_scr}) initiated at $(w_1^0,w_2^0)\in D\cap \Omega$.
The equation of  $s_3$ (see~(\ref{L_ir}))
has an unique positive solution
$$
w_2\sim \frac{1}{2}(w_1-1)^{-1/2} \quad \mbox{as} \quad w_1 \to 1+0.
$$
Therefore, we have $\alpha=1/2$ in Proposition \ref{asymp3}.
Since  $\frac{1-2a}{4a}>\alpha=1/2$ whenever $0<a<1/4$
the trajectory $\bigl(w_1(t), w_2(t)\bigr)$ lies over the curve $s_3$ for $w_1 \to 1+0$ (corresponding to  $t\to +\infty$).
By continuity  there exists a point on the curve  $s_3\cap \Omega$ at which $\bigl(w_1(t), w_2(t)\bigr)$ intersects $s_3\cap \Omega$
and leaves the set $D$.

Finally,  we see that for initial points close to the point of the type
$(w_1,w_2)\in c_i'$, $i=1,2,3$,
the time for leaving the set of metrics with positive sectional curvature could be as long as we want.
\end{proof}

\smallskip

Let us consider the  vector field $V:=(f,g)$, associated with the  system~(\ref{rnrf_scr}), and
the gradient $\nabla l_i\equiv \left(\frac{\partial l_i}{\partial w_1}, \frac{\partial l_i}{\partial w_2}\right)$ that is the normal vector of the curve $s_i$  (see~(\ref{L_ir})), $i=1,2,3$.

\begin{lemma}\label{not_returns}
No trajectory of the system {\rm(}\ref{rnrf_scr}{\rm)}, $a\in \left\{1/9, 1/8, 1/6\right\}$,  could return back to the domain $D$ leaving $D$ once.
\end{lemma}

\begin{proof} Consider  points  $(w_1,w_2)\in \partial (D)\cap \Omega$ without loss of generality.
It is required to prove that the inequality
$(V,  \nabla l_3)<0$
holds at  every point of the part  $s_3\cap \Omega$
of the curve $s_3$ (in fact the mentioned inequality holds at every point of $s_3$ as we will see below).
Here, $(V,  \nabla l_3)$ means the usual inner product of the vectors $V$ and $\nabla l_3$ in the plane $(w_1,w_2)$.
By direct calculations we get
$$
 (V,  \nabla l_3 )=2(w_1-1)(w_2-1)W,
$$
where $W:=12aw_1^2w_2^2-3(w_1+w_2)w_1w_2(1-2a)-(w_1-w_2)^2(1+2a)$.

Substituting the expression $3w_1^2w_2^2=2(w_1+w_2)w_1w_2+(w_1-w_2)^2$ which is equivalent to $l_3=0$
into $W$ yields
$$
W=(14a-3)(w_1+w_2)w_1w_2-(w_1-w_2)^2(1-2a)<0.
$$

To  complete the proof of the lemma note that
the normal vector $\nabla l_3$ of the curve $s_3$ is inner for the set $D$
since
$$
\frac{\partial l_3}{\partial w_2}=-2(w_1-1)(3w_1w_2+(w_2-w_1))<0
$$
on the curve $s_3$ (the curve $s_3$ has no singularities).

\end{proof}

\begin{remark}\label{obrdornet}
Actually, we have proved a more strong assertion in the proof of
Lemma \ref{not_returns}: No one integral curve of the system {\rm(}\ref{rnrf_scr}{\rm)}
initiated outside $D$, could reach the set $D$ (see Figure~\ref{Fig5a}).
In particular, the normalized Ricci flow could not evolve metrics with mixed sectional curvature to metrics with positive sectional curvature.
\end{remark}

Now we are ready to prove Theorem \ref{the2}.

\medskip

{\it Proof of Theorem \ref{the2}\,\,} According to (\ref{siv}),
we can consider the set  $D\setminus \{(1,1)\}$ in the plane $(w_1,w_2)$
instead of the set  (\ref{DPG}) of invariant metrics with positive sectional curvature
as it was noted in Remark~\ref{D_rema}.
Now, it suffices to apply Lemmas     \ref{attains} and \ref{not_returns} to complete the proof of the theorem
and the additional assertions just after Theorem \ref{the2}.

\medskip

\section{Evolution of invariant metrics with positive Ricci curvature}\label{on_R_set}

Let us describe the set $R$ of invariant metrics with positive Ricci curvature on the given Wallach space.
Since the principal Ricci curvatures ${\bf r_i}$ are expressed as  $\frac{x_jx_k+a(x_i^2-x_j^2-x_k^2)}{2x_1x_2x_3}$, we consider  the functions
$$
k_i:=x_jx_k+a(x_i^2-x_j^2-x_k^2),
$$
where
$x_i>0$, \,
$i\ne j\ne k\ne i$, \, $i,j,k \in \{1,2,3\}$.

It is clear that the sets of invariant metrics with non-negative and positive Ricci curvature are respectively the following:
\begin{eqnarray}
\left\{(x_1,x_2,x_3)\in \mathbb{R}_{+}^3\,\, | \,\, k_1 \geq 0,\, k_2 \geq 0,\, k_3 \geq 0 \right\},\label{dom_nonneg_Ricci}\\
\left\{(x_1,x_2,x_3)\in \mathbb{R}_{+}^3 \,\,| \,\, k_1 > 0,\, k_2>0,\, k_3>0 \,\right\}.\label{dom_posit_Ricci}
\end{eqnarray}

Now, consider the description of the domain $R$ in the coordinates $(w_1,w_2)$.
Denote by  $r_i$  curves determined by the equations
$k_i\bigl(\frac{1}{w_1},\frac{1}{w_2},1\bigr)=0$
respectively (see Figure~\ref{Fig3a}). For $w_1>0$ and $w_2>0$, these equations are respectively equivalent to
{\renewcommand{\arraystretch}{1.5}
\begin{equation}\label{r_ir}
\begin{array}{l}
\rho_1:=-aw_1^2w_2^2-aw_1^2+aw_2^2+w_1^2w_2=0,\\
 \rho_2:=-aw_1^2w_2^2+aw_1^2-aw_2^2+w_1w_2^2=0,\\
 \rho_3:=\,\,\,\,\,aw_1^2w_2^2-aw_1^2-aw_2^2+w_1w_2=0.
\end{array}
\end{equation}}
Since the set (\ref{dom_nonneg_Ricci})
is connected and its boundary is a part of
the union of the cones $k_1=0$, $k_2=0$ and $k_3=0$
we easily get on the plane $(w_1,w_2)$  a connected domain $R$ bounded by the  curves
$r_1,r_2$ and $r_3$
solving the system of inequalities $k_i\bigl(\frac{1}{w_1},\frac{1}{w_2},1\bigr)>0$, $i=1,2,3$.

\begin{figure*}[t]
\centering
\includegraphics[angle=0, width=0.45\textwidth]{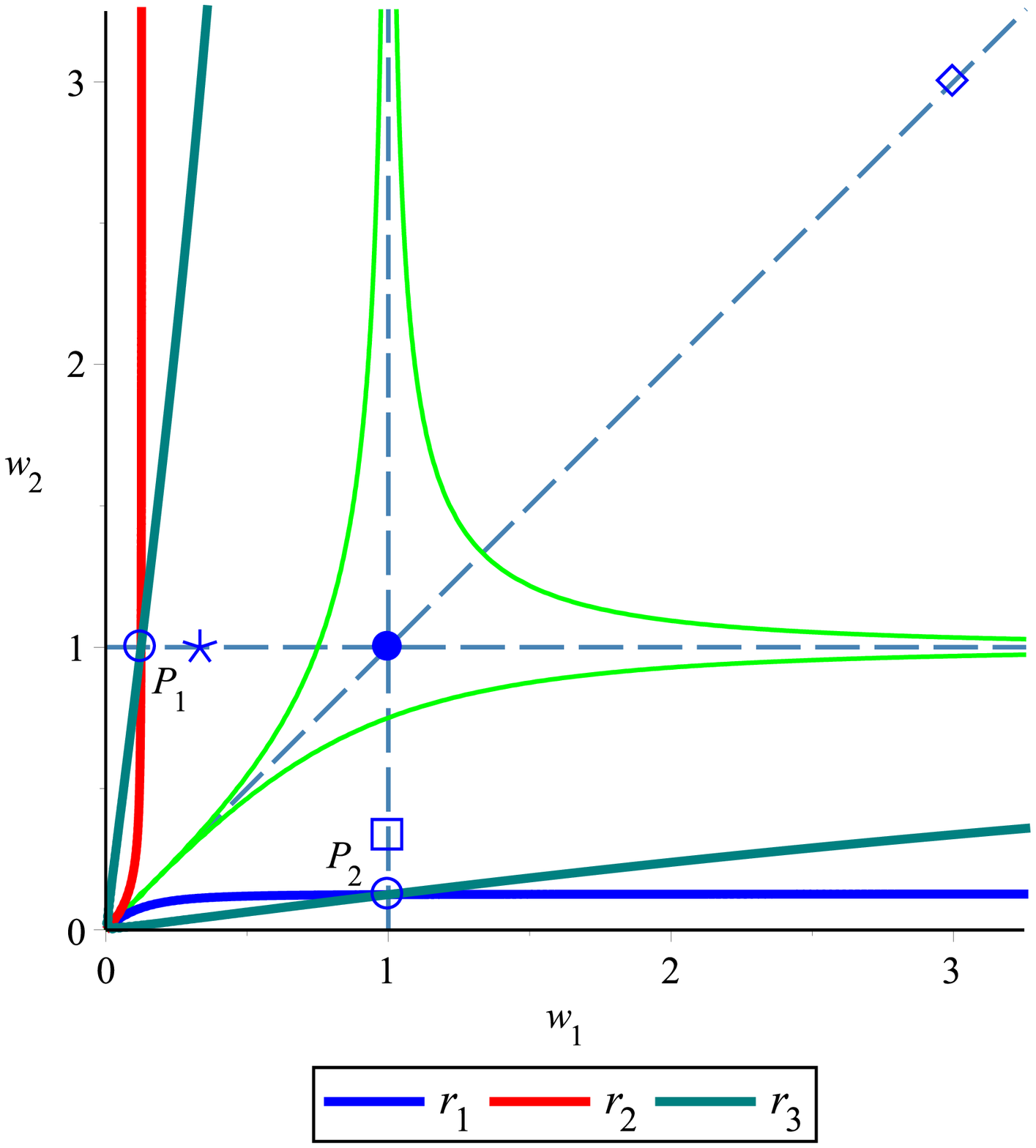}
\includegraphics[angle=0, width=0.45\textwidth]{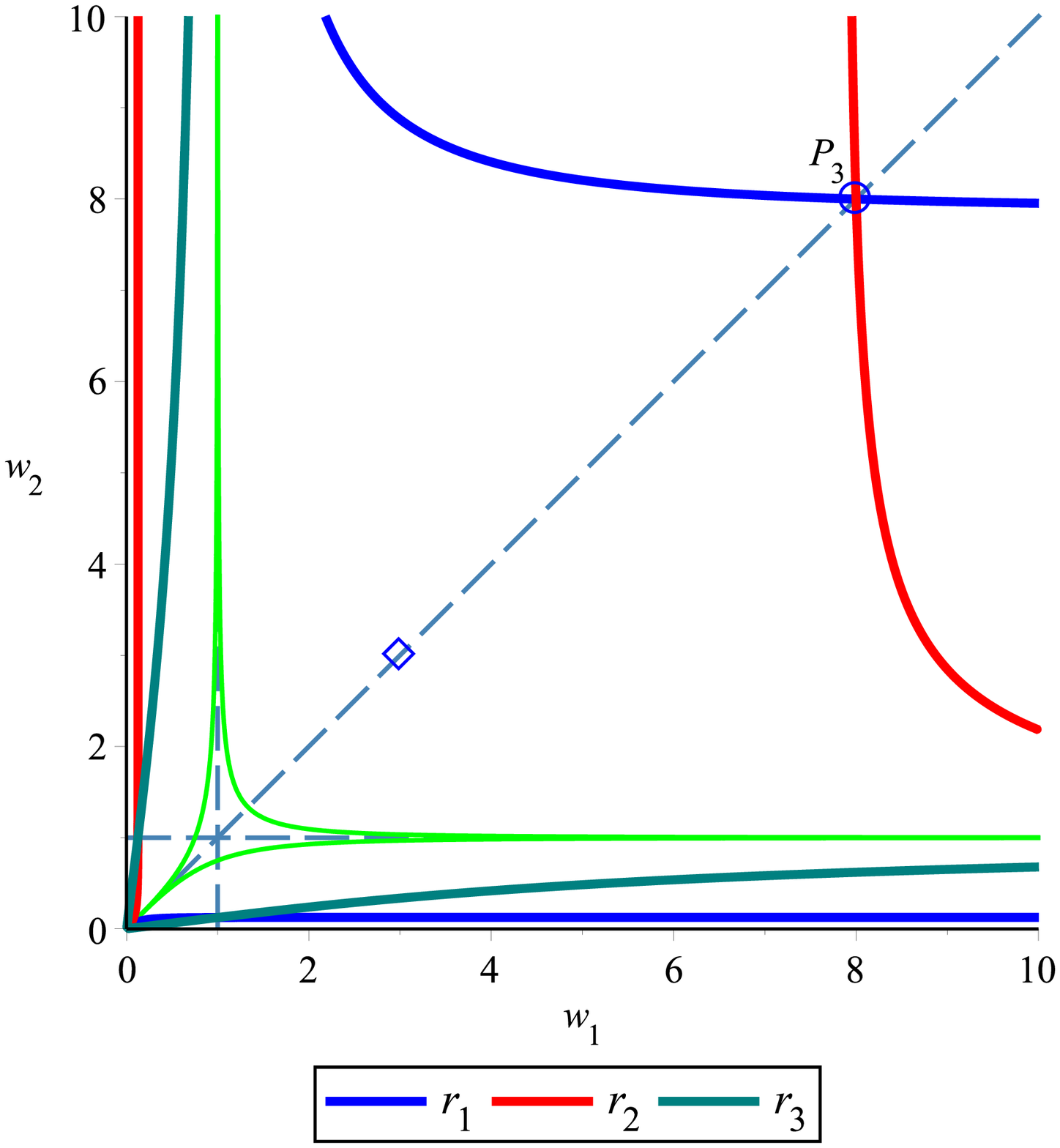}
\caption{The case $a=1/8$: The curves $r_1,r_2,r_3$ and the points $P_1, P_2, P_3$ corresponding to the system (\ref{rnrf_scr}).}
\label{Fig3a}
\end{figure*}

Below we  reveal some useful properties of the curves $r_i$.
It is clear that each of the curves $r_i$,  $i=1,2,3$,  consists of two disjoint connected components.
In general we will use the description of $r_i$'s given
by $k_i\bigl(\frac{1}{w_1},\frac{1}{w_2},1\bigr)=0$,
but  we will concretize the component of $r_i$ in cases when it is necessary.

Let us show that
$r_i\cap s_j=\emptyset$ for $i,j \in \{1,2,3\}$ and $w_1 > 0$,
$w_2 > 0$.
By symmetry, we will confirm the  equality $r_1\cap s_3=\emptyset$ only.
In fact, eliminating~$w_2$ from the system of the equations
$\rho_1=0$ and
$l_3=0$,
we get the quadratic equation
$$
(10a+3)(2a-1)w_1^2-(2a-1)^2w_1-16a^2=0
$$
which has no real solution since its discriminant is negative at $a\in [\frac{1}{9}, \frac{1}{2})$:
$$
(18a-1)(2a-1)(1+6a)^2<0.
$$
Next, easy calculations show (see Figure~\ref{Fig3a})
$$
c_1\cap r_2\cap r_3=\{P_1\}, \quad
c_2\cap r_1\cap r_3=\{P_2\}, \quad
c_3\cap r_1\cap r_2=\{P_3\},
$$
where
\begin{equation}\label{rirr}
P_1:=(a,1),  \quad P_2:=(1,a), \quad P_3:=\left(a^{-1},a^{-1}\right).
\end{equation}

It is easy to see that $c_3$ is tangent to the curves $r_1$ and $r_2$ at the point $(0,0)$, whereas the  pairs
$(r_1,r_3)$ and $(r_2,r_3)$  have the asymptotes  $c_2$ and $c_1$ respectively.

\begin{remark}\label{R_rema}
By analogy with the case of the sectional curvature, it suffices to prove Theorems \ref{Sect_Ricci}, \ref{Sect_Ricci_gen}, and \ref{Sect_Ricci_genn} only
for invariant metrics $\bigl(\frac{1}{w_1},\frac{1}{w_2},1\bigr)$ with $(w_1,w_2)\in R$
in the coordinates $(w_1,w_2)$.
\end{remark}

In what follows, we will need the curves $c_1,c_2$ and $c_3$ introduced in Lemma~\ref{E_i_hyperr}.

\begin{lemma}\label{attains_r_i}
If $a\in (0,1/6)$ then every  integral curve of the system {\rm(}\ref{rnrf_scr}{\rm)}, initiated in
$R\setminus \left(c_1\cup c_2 \cup c_3\right)$, reaches the boundary $r_1\cup r_2\cup r_3$ of $R$ in finite time
and leaves $R$.
This finite time could be as long as we want.
\end{lemma}

The corresponding phase portraits are depicted in Figure \ref{Fig6a}.
\smallskip

\begin{figure*}[t]
\centering
\includegraphics[angle=0, width=0.45\textwidth]{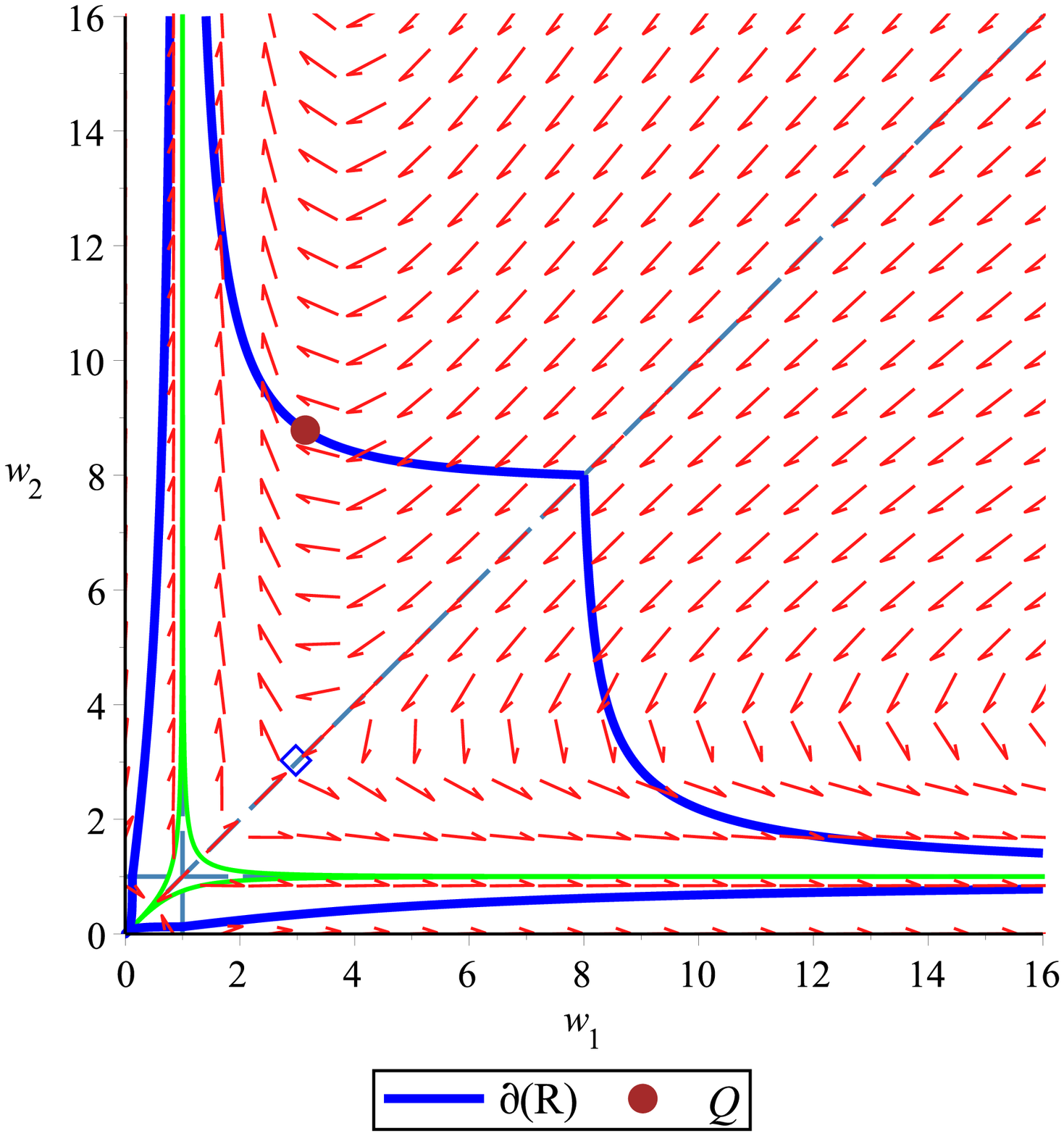}
\includegraphics[angle=0, width=0.45\textwidth]{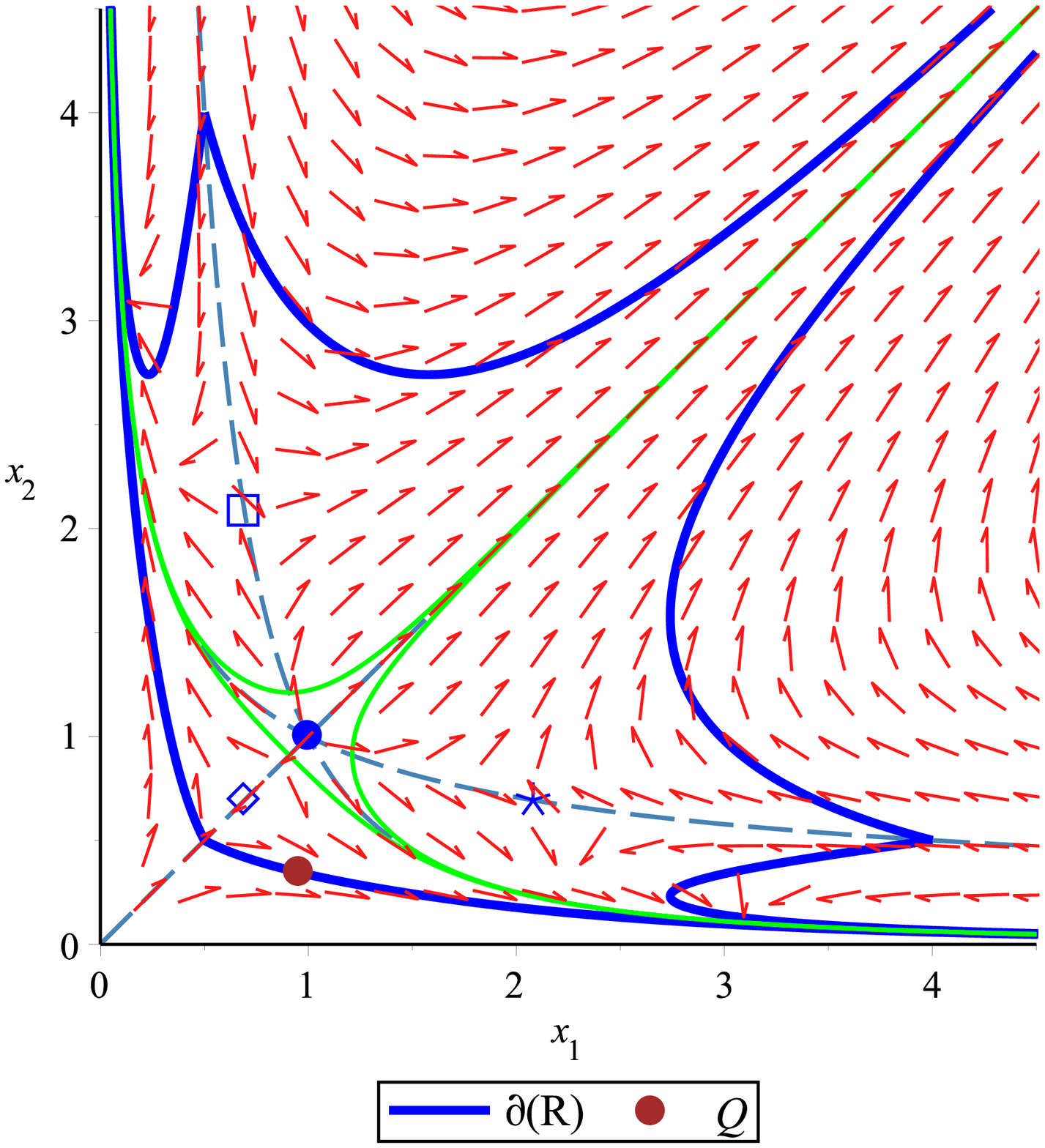}
\caption{The case $a=1/8$: The phase portraits of the systems (\ref{rnrf_scr}) (the left panel) and (\ref{rnrf_sc}) (the right panel).}
\label{Fig6a}
\end{figure*}

\begin{proof}
It is sufficient to consider only the set $R\cap \Omega$, where $\Omega$ given by (\ref{Omegar}).
Consider any trajectory $\bigl(w_1(t), w_2(t)\bigr)$
of the system (\ref{rnrf_scr}) initiated at an arbitrary point $(w_1^0,w_2^0)\in R\cap \Omega$.
The equation $\rho_1=0$ for the curve $r_1$ (see (\ref{r_ir})) has the  solution
$$
w_2\sim \frac{1}{2a}(w_1-1)^{-1} \quad \mbox{as} \quad w_1 \to 1+0\,,
$$
corresponding to the  ``upper'' part $\gamma:=r_1\cap \Omega$  of the curve $r_1$
(see the right panel of Figure~\ref{Fig3a}).
Note that ${\frac{1-2a}{4a}}>1$ for all $0<a<1/6$. Then according to  Proposition~\ref{asymp3}
the trajectory $\bigl(w_1(t), w_2(t)\bigr)$ lies over $\gamma$ for $w_1 \to 1+0$ (corresponding to  $t\to +\infty$).
Hence by continuity there exists a point on $\gamma$  at which $\bigl(w_1(t), w_2(t)\bigr)$
must intersect $\gamma$ and leave $R$.

Finally,  we see that for initial points close to the point of the type
$(w_1,w_2)\in c_i$, $i=1,2,3$,
the time for leaving the set of metrics with positive Ricci curvature could be as long as we want.
\end{proof}

\begin{remark} \label{exccase}
Note that for $a=1/6$ we get the equality ${\frac{1-2a}{4a}}=1$. Hence, the arguments in the above proof do not work for the space
$W_6$. Moreover, we know that Lemma \ref{attains_r_i} is failed for this space, see Theorem 8 of~\cite{ChWal}.
\end{remark}

\begin{remark}\label{pro_dvoe}
The equation of $r_1$ (see (\ref{r_ir})) has also an another solution
$w_2=a + O(w_1-1)$  corresponding to the ``lower'' part
of the curve $r_1$ (see the left panel of Figure~\ref{Fig3a}). Note that in this case we have exactly the point $P_2=(1,a)$ {\rm(}see~{\rm(}\ref{rirr}{\rm)}{\rm)}
as $w_1 \to 1$.
\end{remark}

\begin{lemma}\label{not_returns_to_R}
No trajectory of the system {\rm(}\ref{rnrf_scr}{\rm)}, $a\in \{1/6,1/8,1/9\}$,  could return back to the domain $R$ leaving $R$ once.
\end{lemma}

\begin{proof}
It suffices to prove this lemma for points  $(w_1,w_2)\in \partial (R)\cap \Omega$.
Recall that each of the curves $r_i$ consists of two disjoint connected components.
Therefore we will consider the piece $\gamma:=r_1\cap \Omega$ of the ``upper'' part of the curve $r_1$ which can be
parameterized by the following way
\begin{eqnarray}\label{param_r1}
w_2 = \frac{t+\sqrt{4a^2(1-t^2)+t^2}}{2at}, \quad
w_1 =t w_2, \quad 0<t<1.
\end{eqnarray}

Taking into account (\ref{r_ir}) and (\ref{param_r1}), we get the following inner product:
$$
(V,  \nabla \rho_1 )= \frac{w_2^2}{2a^2}\,W, \quad \mbox{where}
$$
{\renewcommand{\arraystretch}{1.5}
\begin{equation}\label{scalW1}
\begin{array}{l}
W:=\bigl((2-8a^2)t^2+(2a^2+a-1)t+4a^2\bigr)\,\sqrt{4a^2(1-t^2)+t^2} \\
+2(4a^2-1)(2a^2-1)t^3-(a-1)(4a^2-1)t^2\\
-8a^2(2a^2-1)t+2a^2(2a-1).
\end{array}
\end{equation}
}
{\bf Claim 1:} {\it  For every fixed $a\in \{1/9,1/8,1/6\}$ there is an unique point
$Q:=(w_1^{\ast}, w_2^{\ast})\in \gamma$ such that
$(V,  \nabla \rho_1)=0$ at $(w_1^{\ast}, w_2^{\ast})$} (see the left panel of Figure~\ref{Fig6a}).
Indeed for the fixed $a$  we can find  roots  $t^{\ast}=t^{\ast}(a)$ of the equation
$W=0$ which belong to the interval $(0,1)$.
Then the corresponding values of $w_1^{\ast}$ and $w_2^{\ast}$ can be determined
from (\ref{param_r1}).
Thus let us consider the following cases separately.

\smallskip

{\it The case $a=1/9$}. Set $a=1/9$ in (\ref{scalW1}). Then $W=0$ has the unique root
$$
t^{\ast}=\frac{9}{8}+\frac{1}{56}\sqrt{2737}-\frac{1}{88}\sqrt{11242+198\sqrt{2737}}\approx 0.389089209
$$
such that $t^{\ast}\in (0,1)$.
The corresponding values of  $w_1^{\ast}$ and $w_2^{\ast}$ are
\begin{eqnarray*}
w_1^{\ast}&=&w_1(t^{\ast})\approx 3.364907691, \\
w_2^{\ast}&=&w_2(t^{\ast})\approx 8.648165018.
\end{eqnarray*}

\smallskip

{\it The case $a=1/8$}. Then analogously
 \begin{eqnarray*}
t^{\ast}&=&1+\frac{\sqrt{21}}{6}-\frac{1}{10}\sqrt{105+20\sqrt{21}}\approx 0.361437711, \\
w_1^{\ast}&=&w_1(t^{\ast})\approx 3.166087521, \\
w_2^{\ast}&=&w_2(t^{\ast})\approx 8.759704438.
\end{eqnarray*}

\smallskip
{\it The case $a=1/6$}. Then
 \begin{eqnarray*}
t^{\ast}&=&1-\frac{\sqrt{10}}{4}\approx 0.2094305850, \\
w_1^{\ast}&=&w_1(t^{\ast})\approx 2.125323812, \\
w_2^{\ast}&=&w_2(t^{\ast})\approx 10.14810617.
\end{eqnarray*}

\smallskip
Recall now the vertex point $P_3=\left(1/a,1/a\right)$ of the set $R$ introduced in~(\ref{rirr}).
Then it follows that
$$
\max_{a\in \{\frac{1}{9},\frac{1}{8},\frac{1}{6}\}} w_1^{\ast}<\min_{a\in \{\frac{1}{9},\frac{1}{8},\frac{1}{6}\}}{a^{-1}}=6.
$$

{\bf Claim 2:} {\it  $(V,  \nabla \rho_1)<0$ at $1< w_1<w_1^{\ast}$ and
$(V,  \nabla \rho_1)>0$ at $w_1^{\ast}<w_1<1/a$ for every
$a\in \{1/9,1/8,1/6\}$}.
This follows from the fact that the function $w_1(t)$ determined by  (\ref{param_r1}) is monotonically increasing at $t\in (0,1)$, moreover,  \,
$\lim\limits_{t\rightarrow +0}w_1(t)=1$, \,
$\lim\limits_{t\rightarrow 1-0}w_1(t)=1/a$.
Hence by the continuity of the function $\Theta:=(V,  \nabla \rho_1)$ it suffices to
check its sign for representative points chosen from both of the intervals $(0,t^{\ast})$ and $(t^{\ast},1)$ since
$w_1(1)=1/a$.
Indeed, as the calculations show, $\Theta(t^{\ast}-\varepsilon)<0$ and   $\Theta(t^{\ast}+\varepsilon)>0$ for $\varepsilon=10^{-2}$.

\smallskip

{\bf Claim 3:} {\it The normal vector $\nabla \rho_1$
of the curve $\gamma$ is
inner for the set $R$ for all $w_1>1$}.
Note that $w_2>1/a$ for the considered ``upper'' part $\gamma$ of the curve $r_1$.
Therefore,
$$
\frac{\partial \rho_1}{\partial w_1}=-2w_1 \bigl( w_2(aw_2-1)+a \bigr)<0.
$$

We proved that trajectories of  the system (\ref{rnrf_scr})
starting from the part of the boundary   $\gamma\subset \partial(R)$
 move towards the set $R$ if $w_1^{\ast}<w_1<1/a$
and move away from $R$ whenever $1<w_1<w_1^{\ast}$. Hence they never can return back to $R$ leaving it once.
\end{proof}

\begin{remark}\label{obrdornet2}
Actually, we have proved a more strong assertion in the proof of
Lemma \ref{not_returns_to_R}:
Some integral curves of the system {\rm(}\ref{rnrf_scr}{\rm)},
initiated outside the domain  $R$, could reach  $R$ {\rm(}e.~g. through the part of the curve $r_1$ between the points  $P_3$ and $Q$,
intersecting $r_1\subset \partial(R)$ from up to down{\rm)}, see the left panel of Figure~\ref{Fig6a}.
But later these trajectories will leave $R$ irrevocably,
if will reach   $\partial(R)$ {\rm(}e.~g. in $R\cap \Omega$, this could happen about
the part of $r_1$ situated from the left of the point $Q${\rm)}.
Note that this effect  follows also from Lemma \ref{attains_r_i} for $a=1/8$ and $a=1/9$.
Hence, in particular, the normalized Ricci flow can evolve some metrics with mixed Ricci curvature
to metrics with positive Ricci curvature.
\end{remark}

{\it Proof of Theorem \ref{Sect_Ricci}\,\,} According to (\ref{siv}), we can consider the set
$R$ instead of the set  (\ref{dom_posit_Ricci}) of invariant metrics with positive Ricci
curvature
as it was noted in Remark \ref{R_rema}.
Now, it suffices to apply Lemmas     \ref{attains_r_i} and \ref{not_returns_to_R} to complete the proof of the theorem
and the additional assertions just after Theorem \ref{Sect_Ricci}.

\smallskip

{\it Proof of Theorem \ref{Sect_Ricci_gen}\,\,}
It is sufficient to work with  the set $\Omega$ given by (\ref{Omegar}).
The equation $\rho_1=0$ for the curve $r_1$ (see (\ref{r_ir})) has the  solution
$$
w_2\sim \frac{1}{2a}(w_1-1)^{-1} \quad \mbox{as} \quad w_1 \to 1+0\,,
$$
corresponding to the  ``upper'' part $\gamma$  of the curve $r_1$, which is the ``upper'' part of the boundary of
$R\cap \Omega$, the set of metric with positive Ricci curvature in $\Omega$
(see the right panel of Figure~\ref{Fig3a}).

Consider the case $a\in (0,1/6)$ and any trajectory $\bigl(w_1(t), w_2(t)\bigr)$
of the system~(\ref{rnrf_scr}) initiated at a point of $R\cap \Omega$.
Note that ${\frac{1-2a}{4a}}>1$ for all $0<a<1/6$. Then according to  Proposition~\ref{asymp3}
the trajectory $\bigl(w_1(t), w_2(t)\bigr)$ lies over $\gamma$ for $w_1 \to 1+0$ (corresponding to  $t\to +\infty$).

Now, consider the case $a\in (1/6,1/4)\cup (1/4,1/2)$.
Clearly, ${\frac{1-2a}{4a}}<1$ for all $a\in (1/6,1/2)$. Proposition~\ref{asymp3}
implies that
the normalized Ricci flow evolves every initial metric in $\Omega$ into metrics with positive Ricci curvature.
This proves the theorem.
\smallskip

\begin{figure*}[t]
\centering
\includegraphics[angle=0, width=0.45\textwidth]{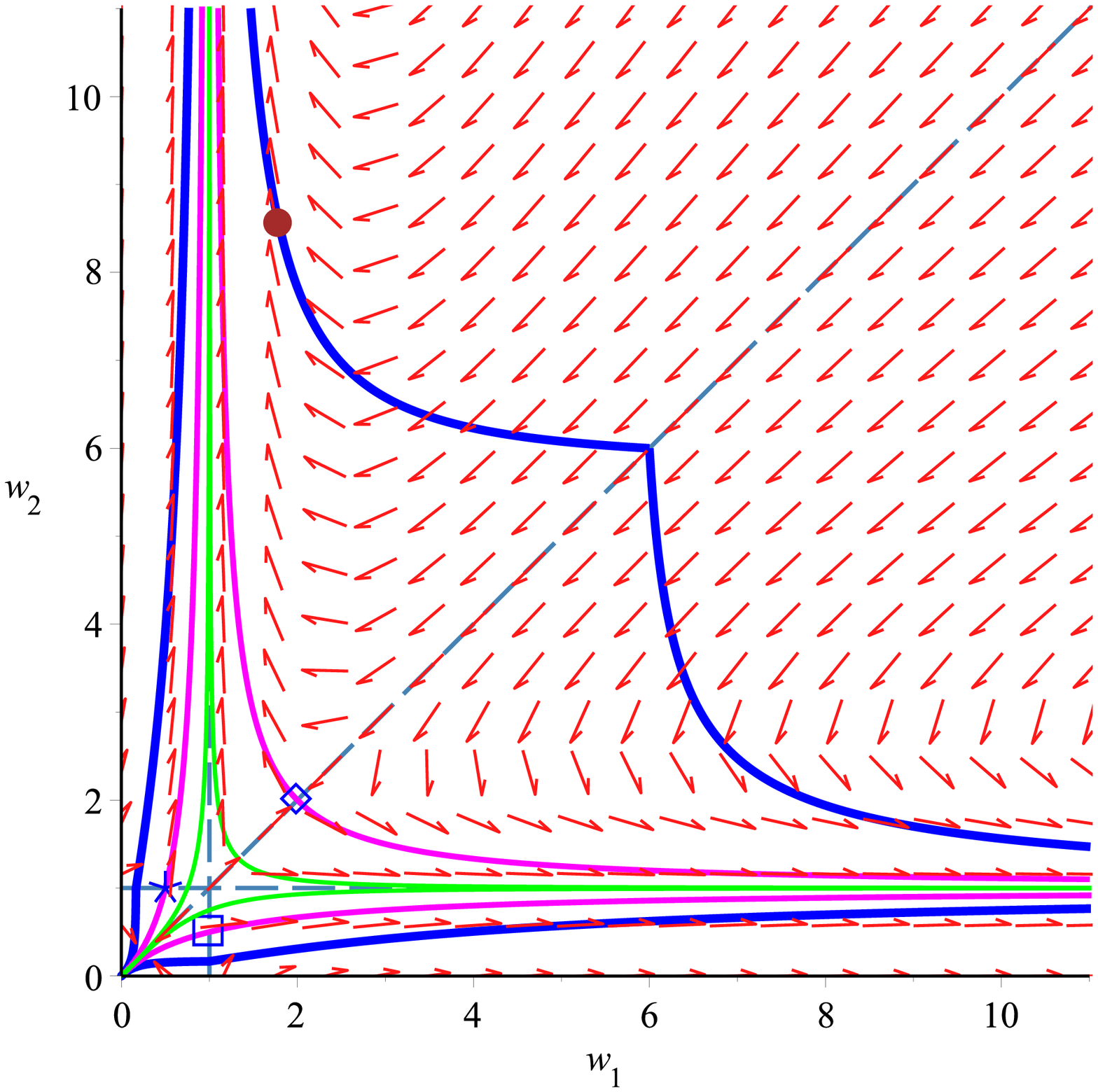}
\includegraphics[angle=0, width=0.45\textwidth]{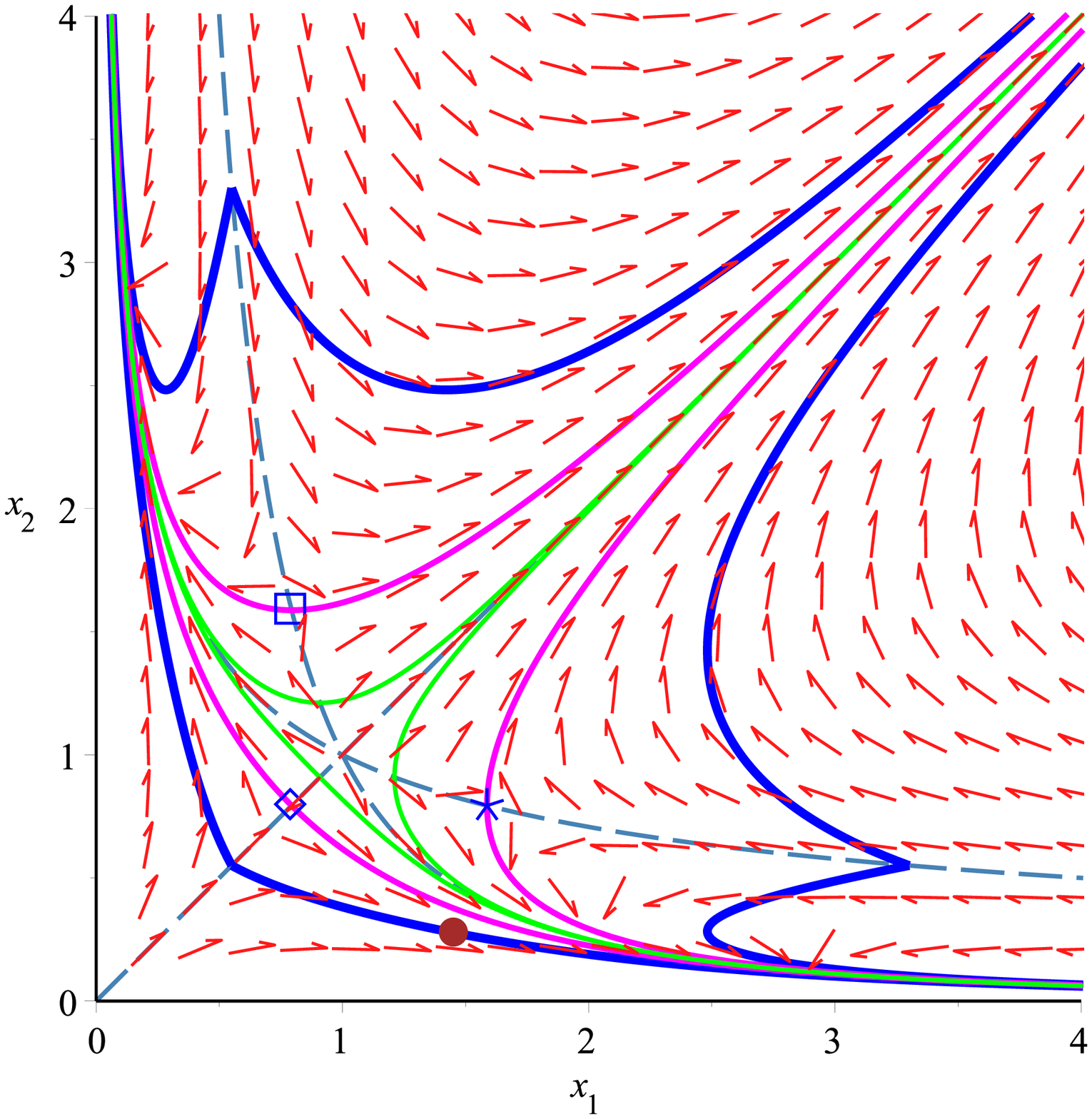}
\caption{
The case $a=1/6$: The domains of positive sectional and positive Ricci curvatures, K\"{a}hler metrics, the phase portraits of the systems (\ref{rnrf_scr})
(the left panel) and (\ref{rnrf_sc}) (the right panel).}
\label{Fig4z}
\end{figure*}

{\it Proof of Theorem \ref{Sect_Ricci_genn}\,\,}
First, note that the set of metrics with  the property  $x_i=x_j+x_k$ is an invariant set
of the system (\ref{nrf_sc}) with right hand sides $F_i:=-2x_i(t)\left({\bf r_i}-\frac{S}{n}\right)$
for $a=1/6$.
Indeed, if we consider any metric with $x_3=x_1+x_2$, then
direct calculations show that $F_1+F_2-F_3\equiv 0$ for $a=1/6$. Note also that every  non-normal Einstein metric
on the space under consideration is such that $x_i=x_j+x_k$ for suitable indices.

Hence, in the scale invariant coordinates $(w_1,w_2)$
we have an invariant curve  $w_1^{-1}+w_2^{-1}=1$ of the system (\ref{rnrf_scr})
passing through the point $E_3=(2,2)$, see the left panel of Figure~\ref{Fig4z}.
Since $E_3$ is a saddle of the system (\ref{rnrf_scr}),
the curve $w_1^{-1}+w_2^{-1}=1$ is necessarily one of the separatrices (more exactly, the unstable manifold) of
this point $E_3$ by uniqueness of a solution of the initial value problem (obviously the line $w_2=w_1$ is the second separatrix).

For submersion metrics the proof is easy and follows from the discussion in Introduction. Let us consider the  case of generic metrics.
Without loss of generality we may suppose that  the initial metric is in $\Omega$.
By the above discussion, the set $\left\{(w_1,w_2)\,|\,w_2<\frac{w_1}{w_1-1}\right\}\cap \Omega$ is an invariant set of  the system~(\ref{rnrf_scr}).
Simple calculations show that the curve $\left\{(w_1,w_2)\,|\,w_2=\frac{w_1}{w_1-1}\right\}\cap \Omega$  lies under the curve
$r_1\cap \Omega\subset \partial(R)$.
Hence, every trajectory of  (\ref{rnrf_scr})
initiated in the set $\left\{(w_1,w_2)\,|\,w_2<\frac{w_1}{w_1-1}\right\}\cap \Omega$ remains in the domain $R\cap\Omega$, that proves the theorem.

\begin{remark}
For $W_6$, the metrics \eqref{metric} with $x_i=x_j+x_k$ constitute the set of K\"{a}hler invariant metrics, see Figure~\ref{Fig4z} and
e.~g. \cite[Chapter 8]{Bes}. The general result that the set of K\"{a}hler  metrics is invariant under the Ricci flow on every manifold
is obtained in \cite{Ban}.
\end{remark}

\begin{remark}
Note that conditions of Theorem \ref{Sect_Ricci_genn} are valid for metrics from $D$, the set of metrics with positive sectional curvature on the space $W_6$.
Hence, we get the generalization of Theorem 8 in~\cite{ChWal}.
\end{remark}

\begin{figure*}[t]
\centering
\includegraphics[angle=0, width=0.45\textwidth]{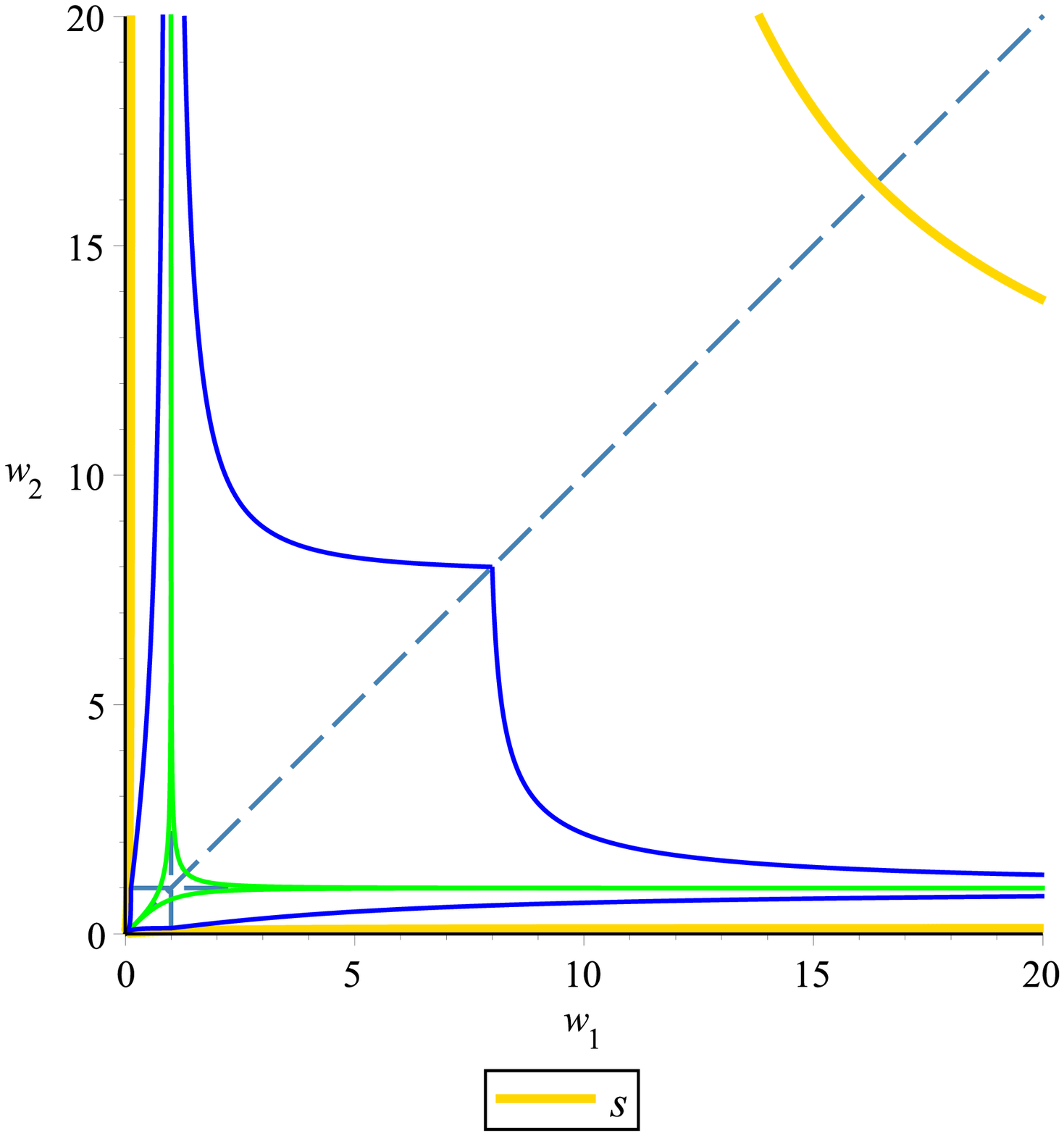}
\includegraphics[angle=0, width=0.45\textwidth]{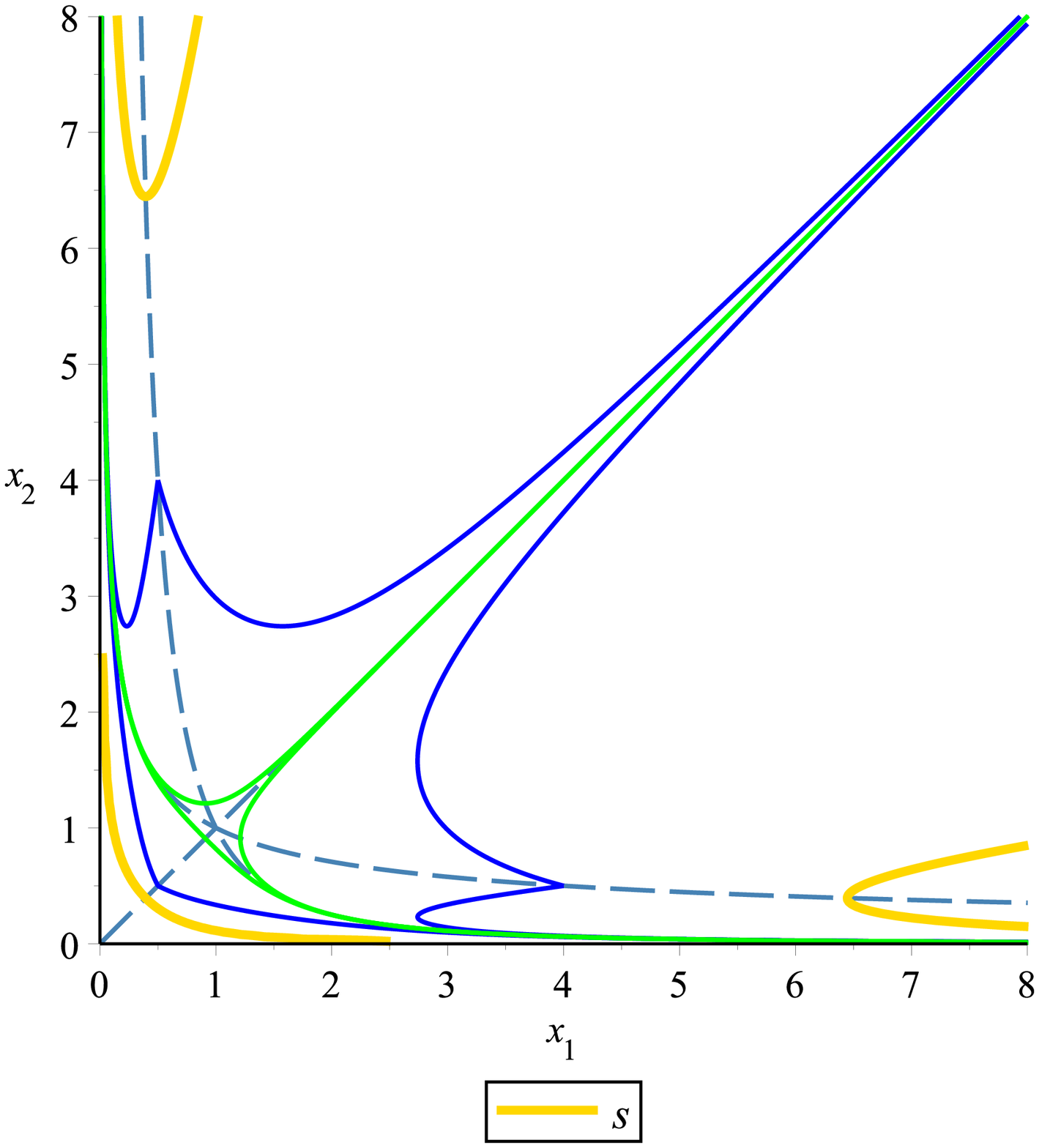}
\caption{The case $a=1/8$: The curve $s$ for the systems (\ref{rnrf_scr}) (the left panel) and (\ref{rnrf_sc}) (the right panel).}
\label{Fig7a}
\end{figure*}

\section{Evolution of invariant  Riemannian metrics with positive scalar curvature and concluding remarks}

For completeness of the exposition, we discuss shortly the evolution of the scalar curvature under
normalized Ricci flow.
We have the following general result related to the evolution of $G$-invariant metrics on a homogeneous space $G/H$ under
the normalized Ricci flow.

\begin{proposition}[\cite{Ham2},\cite{Lauret}]\label{csalinc}
Let $(M=G/H,g_0)$ be a Riemannian homogeneous space.
Consider the solution of the normalized Ricci flow {\rm(}\ref{ricciflow}{\rm)} on $M$ with $\bold{g}(0)=g_0$.
Then
$$
\frac{\partial S}{\partial t} = 2\, \| {\Ric}_\bold{g} \|^2 - \frac{2}{n} \cdot S^2,
$$
where
$S=S(t)$ is the scalar curvature of metrics $\bold{g}(t)$ and $n=\dim(M)$. In particular, the scalar curvature
$t \mapsto S(t)$ increases unless $g_0$ is Einstein.
\end{proposition}

\begin{figure*}[t]
\centering
\includegraphics[angle=0, width=0.45\textwidth]{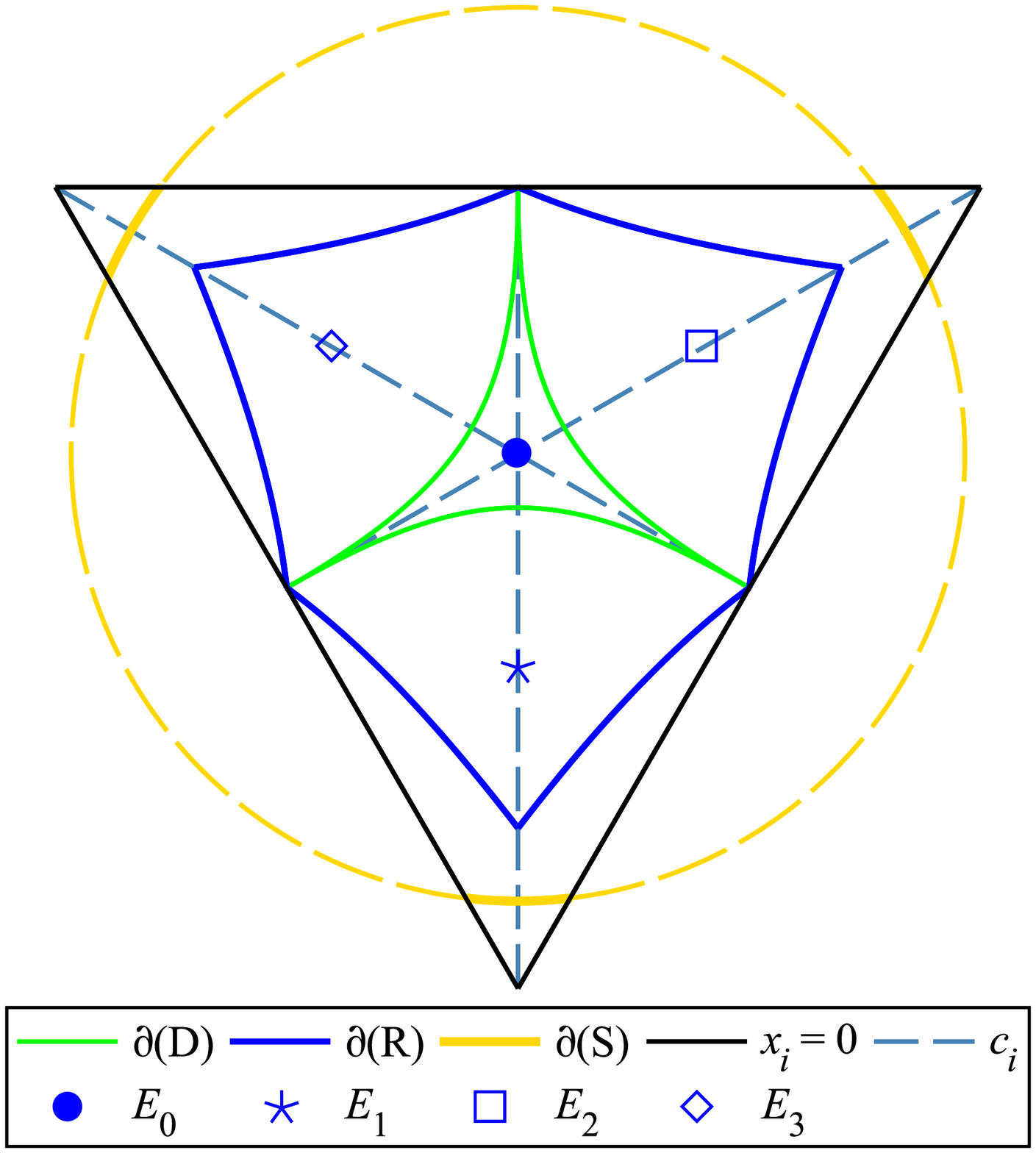}
\includegraphics[angle=0, width=0.45\textwidth]{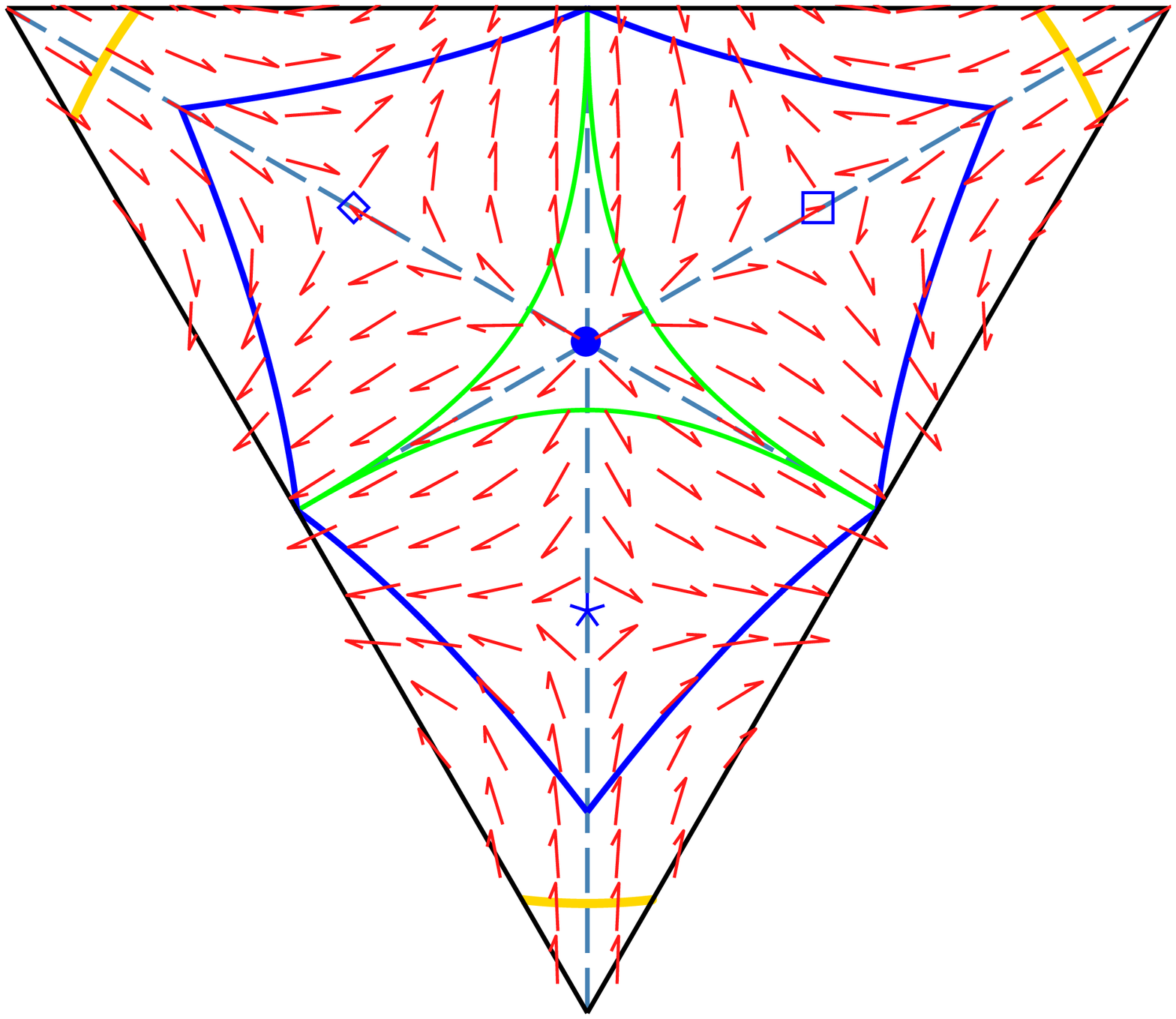}
\caption{The case $a=1/8$: The domains of positive  sectional, positive Ricci, and positive scalar curvatures,
the phase portrait of the system (\ref{nrf_sc})  in the plane $x_1+x_2+x_3=1$.}
\label{Fig1z}
\end{figure*}

\begin{figure*}[t]
\centering
\includegraphics[width=0.5\textwidth]{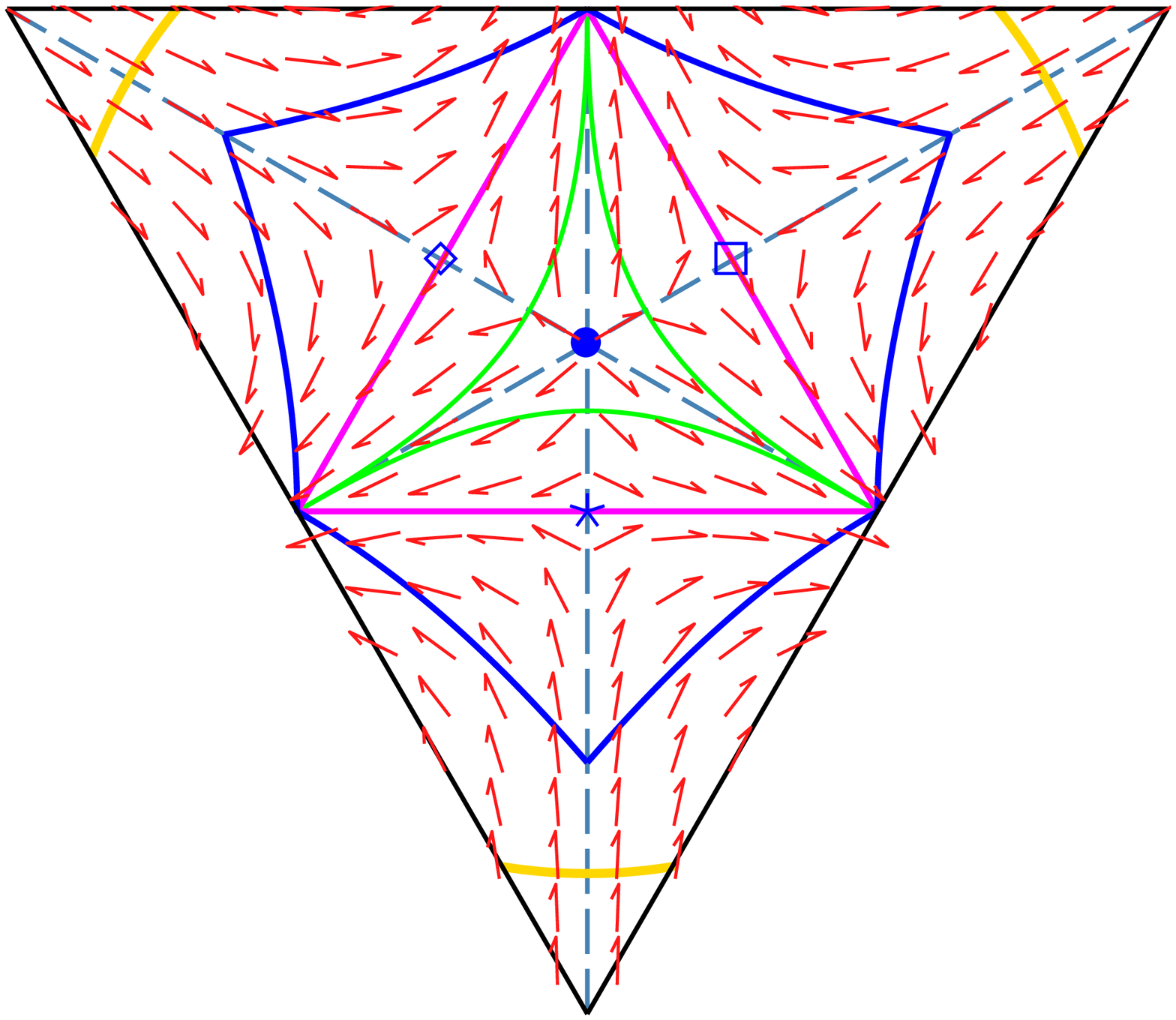}
\caption{The case $a=1/6$: The domains of positive sectional, positive Ricci, and positive scalar curvatures, K\"{a}hler metrics,
the phase portrait of the system (\ref{nrf_sc})  in the plane $x_1+x_2+x_3=1$.}
\label{Fig3z}
\end{figure*}

Therefore, we see that the normalized Ricci flow (on every compact homogeneous space)
with an invariant Riemannian metric of positive scalar curvature as the initial point,
do not leave the set of the metrics with positive scalar curvature.
For the Wallach space $W_{12}$, we reproduce an illustration for this observation in Figure \ref{Fig7a}
(the curve $s$ is the boundary of the set of metrics with positive scalar curvature),
see also Figure \ref{Fig6a} for the corresponding phase portraits.
Note, that the curve
$s$ satisfies the equation
$$
a \left( w_1^2w_2^2+w_1^2+w_2^2\right)-w_1^2w_2-w_1w_2^2-w_1w_2=0.
$$

For compact homogeneous spaces, the integral flow of the scalar curvature
functional on the set of invariant metrics of fixed volume
coincides with the Ricci flow.
Important results
on the behavior of the scalar curvature and good pictures are obtained in
\cite{BWZ} (see also references therein). It is a good option for a reader to compare illustrations and discussions from that paper with
our results.

Finally, we reproduce additional illustrations suggested us by Wolfgang Ziller.
We draw our pictures for the system (\ref{nrf_sc}) in the plane $x_1+x_2+x_3=1$.
These pictures preserves the dihedral symmetry of the initial problem.
We reproduce  in Figure \ref{Fig1z} the domains of positive sectional, positive Ricci, and positive scalar curvatures
(we denote them by $D$, $R$, and $S$ respectively) of the system (\ref{nrf_sc}) in the plane $x_1+x_2+x_3=1$
for $a=1/8$. We also reproduce the phase portrait (of the tangent component) for the system (\ref{nrf_sc}).
Note that Riemannian metrics constitute a triangle and the set $S$ is bounded by a circle.

Similar pictures could be produced for $a=1/9$ and $a=1/6$. We reproduce here only Figure~\ref{Fig3z} (compare with Figure~\ref{Fig4z})
for $a=1/6$, because the space $W_6$ admits K\"{a}hler invariant metrics, that constitute a small triangle in Figure~\ref{Fig3z}.
Note also that three non-normal Einstein metrics in this case are K\"{a}hler~--~Einstein and
one can easily get main properties of the K\"{a}hler~--~Ricci
flow on the space $W_6$ using this picture.

\bigskip

{\bf Acknowledgements.}
The authors are grateful to the anonymous referee
for helpful comments and suggestions that improved the presentation of this paper.
The authors are indebted to Prof. Christoph~B\"{o}hm, to Prof. Nolan~R. Wallach, and to Prof. Wolfgang~Ziller
for helpful discussions concerning this paper.
The project was  supported by Grant  1452/GF4 of Ministry of Education and Sciences of the Republic of Kazakhstan for 2015-2017.

\bibliographystyle{amsunsrt}

\vspace{5mm}

\end{document}